\documentclass[12pt,a4paper]{amsart}

\usepackage[
url=false,
giveninits=true,
isbn=false,
doi=false,
maxbibnames=99,
style=numeric,
sorting=nyt,
maxnames=4,
maxalphanames=4,
sortcites=none
]{biblatex}
\usepackage{amssymb}
\usepackage{tikz}
\usepackage{enumerate}
\usepackage{mathtools}
\usepackage{xcolor}
\usepackage{textcomp}
\usepackage{amsthm}
\usepackage{thmtools} 

\usepackage[
colorlinks=true,
linkcolor=red!50!black,
citecolor=green!50!black,
urlcolor=blue!50!black,
hypertexnames=false
]{hyperref}
\usepackage[capitalise]{cleveref} 

\declaretheoremstyle[headformat=\NUMBER,headpunct=,]{claim}

\declaretheorem[name=Theorem,numberwithin=section,
refname={Theorem,Theorems}]{theorem}

 \newtheorem*{thmA}{Theorem A}
  \newtheorem*{thmC}{Theorem C}
 \newtheorem*{thmB}{Theorem B}

\declaretheorem[name=Lemma,sibling=theorem,refname={Lemma,Lemmas}]{lemma}
\declaretheorem[name=Corollary,sibling=theorem,refname=
{Corollary,Corollaries}]{corollary}

\declaretheorem[name=Proposition,sibling=theorem,refname=
{Proposition,Propositions}]{proposition}

\declaretheorem[name=Definition,sibling=theorem,refname=
{Definition,Definitions}]{definition}

\declaretheorem[name=Example,sibling=theorem,refname={Example,Examples}]{example}
\declaretheorem
[name=Remark,sibling=theorem,refname={Remark,Remarks}]{remark}

\declaretheorem[style=claim,numberwithin=theorem,refname={,}]{claim}

\newcommand{\wt}{\widetilde}

\DeclareMathOperator{\Aut}{Aut}

\DeclareMathOperator{\Inn}{Inn}
\DeclareMathOperator{\Inndiag}{Inndiag}
\DeclareMathOperator{\Out}{Out}
\DeclareMathOperator{\Sol}{Sol}

\DeclareMathOperator{\Syl}{Syl}
\DeclareMathOperator{\syl}{Syl}

\DeclareMathOperator{\M}{M}
\DeclareMathOperator{\HN}{HN}
\DeclareMathOperator{\HS}{HS}
\DeclareMathOperator{\Fi}{Fi}
\DeclareMathOperator{\He}{He}

\DeclareMathOperator{\Suz}{Suz}

\DeclareMathOperator{\Co}{Co}

\DeclareMathOperator{\B}{B}
\DeclareMathOperator{\D}{D}
\DeclareMathOperator{\E}{E}
\DeclareMathOperator{\F}{F}
\DeclareMathOperator{\G}{G}
\DeclareMathOperator{\GL}{GL}

\DeclareMathOperator{\PGL}{PGL}

\DeclareMathOperator{\PSL}{PSL}

\DeclareMathOperator{\PSp}{PSp}
\DeclareMathOperator{\PSU}{PSU}
\DeclareMathOperator{\POmega}{P \Omega}
\DeclareMathOperator{\SL}{SL}

\DeclareMathOperator{\Sp}{Sp}
\DeclareMathOperator{\SU}{SU}

\DeclareMathOperator{\diag}{diag}

\DeclareMathOperator{\Alt}{Alt}
\DeclareMathOperator{\Dih}{Dih}
\DeclareMathOperator{\Sym}{Sym}

\DeclareMathOperator{\GF}{GF}

\DeclareMathOperator{\Soc}{Soc}

\renewcommand{\epsilon}{\varepsilon}

\newcommand{\ov}{\overline}
\newcommand{\wh}{\widehat}

\newcommand{\eps}{\epsilon}

\DeclarePairedDelimiter{\gen}{\langle}{\rangle}

\renewbibmacro{in:}{}
\DeclareFieldFormat[article]{title}{#1}

\title[Invariably generated groups]{The structure of finite groups invariably generated by two elements of prime order}
\author{Inna   Capdeboscq}\author{Chris Parker}

\address{Inna   Capdeboscq\\
Mathematics Institute\\
Zeeman Building\\
University of Warwick\\
Coventry CV4 7AL\\United Kingdom}
\email{I.Capdeboscq@warwick.ac.uk}
\address{Chris Parker\\
School of Mathematics\\
University of Birmingham\\
Edgbaston\\
Birmingham B15 2TT\\
United Kingdom}
\email{c.w.parker@bham.ac.uk}

\date{\today}
\begin{document}
\maketitle

\begin{abstract}
A group $G$ is invariably generated by   two elements $a$ and $b$ if $G=\langle a^g,b^h\rangle $ for all $g,h\in G$. This paper  provides a  structural description of finite groups invariably generated by two elements of distinct prime orders $s$ and $t$.
We   illustrate the use of our main result in a case study
with $s=2$ and $t=3$. In particular we prove that an almost simple  group is invariably  generated by an element of order $2$ and an element of order $3$  if and only if it is isomorphic to $\PGL_2(3^{2^b})$ for some $b \ge 1$.   
\end{abstract}

\section{Introduction}

In finite group theory, especially in the study of finite simple groups, generation results have been a focus of attention since the discovery  that every finite simple group is generated by two elements   \cite{aschbacher1984some,steinberg1962generators}.

\begin{definition}\label{def1}
 Let $n \ge 1$ be a natural number and $\Lambda=  (\lambda_1,\dots,\lambda_n )$ be a sequence    of positive integers. A group $G$ is \emph{invariably $\Lambda$-generated} if and only if  there exists   $(x_1, \dots ,x_n) \in G^n$ with $x_i$ of order dividing $\lambda_i$ such that $$G=\langle x_1^{g_1},\dots,x_n^{g_n}\rangle$$ for all $g_1,\dots, g_n \in G$. We call $(x_1, \dots, x_n)$ an \emph{invariable  $\Lambda$-sequence} for $G$.
\end{definition}

Notice that, if $\Lambda$  and $\Lambda^*$ are as in \cref{def1} and $\Lambda^*$  can be reordered to $\Lambda$, then   $G$ is invariably $\Lambda$-generated if and only if $G$ is  invariably $\Lambda^*$-generated. Thus we typically choose  $\lambda_1 \le \dots\le\lambda_n$.

For a   sequence $\Lambda$ as in \cref{def1}, define $$\pi(\Lambda)=\bigcup_{i=1}^n \pi(\lambda_i)$$ where $\pi(\lambda)$ is the set of primes dividing the natural number $\lambda$.

The interest in invariable generation has its origins in computational Galois theory  and has attracted the attention of numerous researchers starting with Dixon in  \cite{MR1180190}.
The work of Kantor, Lubotzky and Shalev \cite[Theorems 1.1 and  1.3]{MR2852243},  motivated by both applications in Galois theory and  computational group theory,  shows that every finite group $G$ can be invariably generated by at most $\log_2|G|$ elements and every non-abelian finite simple group is invariably generated by two elements.

  Detomi, Lucchini and   Roney-Dougal in \cite{MR3320229} introduced and studied \emph{coprime invariable generation}  where the elements of $\Lambda$ are pairwise coprime.  Tightening the requirements, Detomi and  Lucchini \cite{DETOMI2015683}
introduced \emph{prime-power coprime invariable generation}, where the elements of $\Lambda$ are additionally required to be prime powers. These developments suggest studying invariable generation under increasingly restrictive conditions on the members of $\Lambda$. Indeed in
  \cite[Section 6, Conjecture]{MR2901064}
Dolfi, Guralnick,  Herzog  and Praeger conjecture that, with finitely many exceptions, for a finite simple group of Lie type $G$ there  exist   primes $s<t$ such that $G$ is invariably $(s,t)$-generated. This remains unresolved.
Guralnick, Shareshian and Woodroofe \cite{MR4779375} have   shown that, with four exceptions, a finite simple group of Lie type in characteristic $p$ is invariably generated by a Sylow $p$-subgroup   and an element of prime order.

In this paper, we view  invariable generation from a different perspective.
We first derive a  theorem that describes the structure of a group which is invariably $(s,t)$-generated  with $s<t$   primes. As a case study we then investigate what can be said about an invariably $(2,3)$-generated group.

Recall that for a group $X$, the \emph{socle} of $X$, $\Soc(X)$, is the subgroup of $X$ generated by all the minimal normal subgroups of $X$. We say that $X$ is \emph{almost simple} if $\Soc(X)$ is a non-abelian simple group. Define a group  $X$ to be \emph{almost semisimple}  if and only if  every minimal normal subgroup of $X$ is a non-abelian simple group.  Notice that these definitions imply that the trivial group is almost semisimple whereas an almost simple group is non-trivial.

\begin{definition}\label{def:AG}
For a group $G$, define   $\mathcal{A}(G)$ to be the set of normal subgroups $K$ of $G$ such that $G/K$ is almost semisimple. Then define the \emph{almost semisimple   residual} of $G$ to be
$$ A(G)=  \bigcap_{K\in \mathcal A(G)} K.$$
\end{definition}
As the name suggests, $A(G) \in \mathcal A(G)$ and so the quotient   $G/A(G)$ is almost semisimple (see \cref{lem:almost residual}). We write $$\mathcal A^*(G)=\{K \unlhd G\mid G/K \text { is almost simple}\}.$$
In \cref{lem:AG,prop: G/A(G)}, we show that $A(G) = \bigcap_{K\in \mathcal A^*(G)\cup\{G\}}K$ and that $\Soc(G/A(G))$ is the direct product of the simple groups that appear as the minimal normal subgroups of $G/K$ for $K\in \mathcal A^*(G)$.  This claim requires that for $J,K \in \mathcal A^*(G)$, we cannot have $J<K$ and this follows from the Schreier property of finite simple groups.

Our main structural result about invariably $(s,t)$-generated groups.

\begin{thmA}\label{thm:thm1} Suppose that $s$ and $t$ are prime numbers with $s< t$ and that $G$ is a finite, invariably $(s,t)$-generated group. Then   $G$ is either a soluble $\{s,t\}$-group or $G$  has a normal series $$G\ge K_1> K_2= A(G)\ge K_3\ge 1$$ such that \begin{enumerate} \item $G/K_1$ is a soluble  $\{s,t\}$-group,
\item $G/K_2$ is almost semisimple, $$K_1/K_2 =\Soc(G/K_2)\cong\prod_{L\in \mathcal A^*(G)}\Soc(G/L),$$ and each minimal normal subgroup of $G/K_2$ is a non-abelian simple group of order divisible by $st$,
\item $K_2/K_3$ is a nilpotent group of order coprime to $st$ with $\pi(|K_2/K_3|) \subseteq  \pi(|K_1/K_2|)$, and \item $K_3$ is  a Hall $\{s,t\}$-subgroup of $K_2$.
\end{enumerate}
In particular, $A(G)$ is soluble and, for $L \in\mathcal A^*(G)$, $G/L$ is an invariably $(s,t)$-generated almost simple group.
\end{thmA}

  By \cref{lem:quotients}, quotients of invariably $(s,t)$-generated groups are invariably $(s,t)$-generated and so $G/K_1$, $G/K_2$ and $G/K_3$ are all invariably $(s,t)$-generated.

A noteworthy feature of the proof of Theorem A is that apart from  using the Schreier property, it is  independent of the classification of the finite simple groups. In particular, the Schreier property  yields the quotient $G/K_1$ is a soluble group and so an $\{s,t\}$-group by \cref{lem:Sylow/Hall}.

That Theorem A genuinely  requires $s$ and $t$ to be primes  is demonstrated by part (i) of the following example. Part  (ii) of the example  shows that when $G$ is invariably generated by three elements of distinct prime orders, $A(G)$ need not be soluble. These examples were verified using {\sc magma} \cite{Magma}.

\begin{example}
\begin{enumerate}
\item The wreath product $W=\Alt(5) \wr \Alt(5)$ is invariably $(3,25)$-generated while $A(W)\cong \Alt(5)^5$, the base group of the wreath product,   is not soluble. In particular, by Theorem A, $W$ is not $(3,5)$-invariably generated.
    \item The wreath product $W=\Alt(5) \wr \Sym(2)$ is invariably $(2,3,5)$-generated while $W=A(W)$ is not soluble.
    \end{enumerate} \end{example}

Our interest in proving a result like Theorem A  originally arose from a question of  Pavel Zalesskii   \cite[Question 21.142]{khukhro2026unsolvedproblemsgrouptheory}: ``Let $p \ne q$ be fixed primes. Does   every finite
group embed into a finite group invariably generated by an element of order $p$ and an
element of order $q$?''

The following striking result, which follows easily from Theorem A, sheds considerable light on the answer to Zalesskii's question (where we have replaced $p$ and $q$ by $s$ and $t$).

\begin{thmB}\label{thm:cor1}
 Let $s<t$ be prime numbers.  Suppose that $T$ is a non-abelian simple group and that $T$ is isomorphic to a subgroup of an invariably $(s,t)$-generated group. Then $T$ is isomorphic to a subgroup of an almost simple invariably $(s,t)$-generated group.
 \end{thmB}

Given Theorem A, it is natural to ask what can be said if we fix $(s,t)$. We answer this when $(s,t)=(2,3)$, and here we obtain a complete description.

\begin{thmC}\label{thm:main theorem}
Suppose that $G$ is an invariably $(2,3)$-generated group.
Then either $G=A(G)$ is a soluble $\{2,3\}$-group, or every almost simple quotient of $G$ is isomorphic to $\PGL_2(3^{2^a})$ for some $a\ge 1$.
In addition, if $G$ is not a soluble $\{2,3\}$-group, then there is a normal series
$$G\ge K_1> K_2= A(G)\ge K_3\ge 1$$ where
\begin{enumerate}
\item $|G/K_1|=2$,
\item $G/K_2$ is almost semisimple, and $$K_1/K_2 \cong \prod_{i=1}^n \PSL_2(3^{2^{a_i}})$$ with $1\le a_1< a_2<\dots <a_n$,
\item $K_2/K_3$ is a nilpotent group of odd order  with \\ $\pi(|K_2/K_3|) \subseteq \pi( 3^{2^{a_n}}-1)\cup \pi(3^{2^{a_n}}+1)$, and
\item $K_3$ is a Hall $\{2,3\}$-subgroup of $K_2$.
\end{enumerate}
Furthermore, the almost simple groups $\PGL_2(3^{2^a})$, $a\ge 1$, are invariably $(2,3)$-generated.
\end{thmC}

Using \cref{lem:dif socles} and Theorem C, we  have the following remark.

\begin{remark} For $1\le a_1< \dots< a_n$, there exists an invariably $(2,3)$-generated almost semisimple group with socle $\prod_{i=1}^n \PSL_2(3^{2^{a_i}})$.
\end{remark}

The following example explicitly shows that the section $K_2/K_3$ in Theorem A (and Theorem C) can be non-trivial. We comment that we used {\sc magma} \cite{Magma} to find a $\GF(5)\PGL_2(9)$-module with non-trivial $2$-cohomology and to construct the extension.

\begin{example} In the non-split extension $X=(5^8)^.\PGL_2(9)$, the  invariable $(2,3)$-sequence for the quotient group $X/O_5(X)\cong \PGL_2(9)$ lifts to an invariable $(2,3)$-sequence for $X$.
\end{example}

An immediate consequence of Theorems~B  and C and the subgroup structure of $\PSL_2(3^{2^a})$ is the following statement.

\begin{corollary}\label{cor:Zalesskii Q}
  Suppose that $T$ is a simple group and that $T$ is isomorphic to a subgroup of an invariably $(2,3)$-generated group. Then $T$ is isomorphic  to $\PSL_2(3^{2^a})$ for some $a \ge 1$ or $\Alt(5)$.
 \end{corollary}

In particular, this result provides a negative answer to the question of Zalesskii. Indeed  $\PSL_3(2)$ is the smallest simple group which does not embed into an invariably $(2,3)$-generated group. \footnote{We have learnt that an answer to the question of Zalesskii has also just  been provided in \cite{YY}.}
We also note that   Theorem C makes a small contribution to the investigation of the  Dolfi, Guralnick,  Herzog  and Praeger conjecture mentioned above.

In \cref{sec:2}, we prove both Theorems A and B as well as the already mentioned results regarding the almost semisimple residual.  Our approach to the proof of Theorem C involves $3$-broad subgroups (see \cref{def:pbroad}), a notion which extends that of broad subgroups introduced by Guralnick and Robinson in \cite{GuralnickRobinson1}. Our investigation of broad and $p$-broad subgroups is in \cref{sec:3} and may contain results of independent interest. Our proof of Theorem C commences in \cref{sec:5} where we start the investigation of almost simple groups $G$ which are invariably $(2,3)$-generated. The theorem to be proved over the subsequent three sections is \cref{thm:Lietype23as}  which says that $G$ is invariably $(2,3)$-generated if and only if $G\cong \PGL_2(3^{2^b})$ for some $b \ge 1$. Let $K=\Soc(G)$. The remainder of \cref{sec:5} shows that $K$ cannot be a sporadic simple group, or, unless it has degree $6$, an alternating group. \cref{sec:6} considers the case in which $K$ is $\PSL_2(r^a)$ and here the examples $\PGL_2(3^{2^b})$ of invariably $(2,3)$-generated groups appear and are shown to be examples. \cref{sec:7} considers groups of Lie type with socle not isomorphic to $\PSL_2(r^a)$, the objective being to prove that they are not invariably $(2,3)$-generated. This section is where, after the small rank groups have been eliminated as candidates for $K$, broad and $3$-broad subgroups play a decisive role. Especially in Subsection \ref{subsec:gen lemmas} their presence is used to  show that if $(z,d)$ is an invariable $(2,3)$-sequence for $G$, then $z\not \in K$ and then, if $K$ has a $3$-broad subgroup, that $d\not \in K$. \cref{thm:Lietype23as} is finally proved at the end of \cref{sec:7} and \cref{sec:8} then draws everything together to prove Theorem C.

Throughout the paper all groups are finite. Our group theoretic notation is either standard or explained in the text. Our notation for the simple groups is either self-explanatory or follows \cite{GLS3}.  For an almost simple group $G$, we let $K=\Soc(G)$ and identify $K=\Inn(K) \le G\le \Aut(K)$.  Thus elements of almost simple groups are automorphisms of its socle.
For groups of Lie type we   follow the conventions from \cite{GLS3}.

\section{The structure of invariably $(s,t)$-generated groups: The proof of Theorems A and B}\label{sec:2}

 The objective of the section is to prove Theorems A and B.
 We start with two  result which were mentioned in the introduction.

\begin{lemma}\label{lem:Sylow/Hall}
Suppose that $G$ is invariably $\Lambda$-generated.
\begin{enumerate}
\item If $p$ is a prime and $\pi(\Lambda)=\{p\}$, then $G$ is a $p$-group.
\item If $G$ is soluble, then $G$ is a $\pi(\Lambda)$-group.
\end{enumerate}
\end{lemma}
\begin{proof}
 (i) Let $(x_1, \dots, x_n)$ be an invariable  $\Lambda$-sequence and $S \in \Syl_p(G)$. Then, as each $x_i$ is a $p$-element, $x_i$ is in some Sylow $p$-subgroup of $G$. Thus Sylow's theorem implies that there exists $g_i\in G$ such that $x_i^{g_i}\in S$. But then $$G= \langle x_1^{g_1},\dots,x_n^{g_n}\rangle\le S\le G$$ and so $G=S$ is a $p$-group.

(ii) In this case $G$   is a soluble group which is invariably $\Lambda$-generated. Let $(x_1, \dots, x_n)$ be an invariable  $\Lambda$-sequence. Then using Hall's theorem, let $H$ be a Hall $\pi(\Lambda)$-subgroup. Then each $x_i$ is conjugate to an element of $H$ and so $G=H$ as claimed.
 \end{proof}

\begin{lemma}\label{lem:quotients} Suppose that $G$ is invariably $\Lambda$-generated. Then every quotient of $G$ is invariably $\Lambda$-generated.
\end{lemma}

\begin{proof} Let $K$ be a normal subgroup of $G$. Suppose that $(x_1, \dots, x_n)$ is an invariable $\Lambda$-sequence for $G$. Then, as $x_iK$ has order dividing the order of $x_i$ and the order of $x_i$ divides $\lambda_i$, the order of $x_iK$ divides $\lambda_i$. Hence, as $\gen{x_iK\mid 1 \le i \le n}=\gen{x_i\mid 1 \le i \le n}/K=G/K$, $(x_1K, \dots, x_nK)$ is an invariable $\Lambda$-sequence for $G/K$.
\end{proof}

 We now prove the results about the almost simple residual. We begin with the following observation.

\begin{lemma}\label{lem:almost residual} The quotient $G/A(G)$ is almost semisimple and, in particular,
$A(G) \in \mathcal A(G)$.
\end{lemma}

\begin{proof} It suffices to show that for $X,Y \in \mathcal A(G)$, $X \cap Y \in \mathcal A(G)$. For the calculation, we may suppose that $X\cap Y=1$. Let $N$ be a minimal normal subgroup of $G$. Then $N$ is non-trivial and so without loss we may suppose that $N \not \le X$. The minimality of $N$ shows that $N\cap X=1$ and so $N \cong NX/X$ as $G$-groups. Since $N$ is a minimal normal subgroup of $G$, $NX/X$ is a minimal normal subgroup of $G/X$. As $G/X$ is almost semisimple $N$ is a non-abelian simple group. Hence $G$ is almost semisimple.
\end{proof}

\begin{lemma}\label{lem:AG} We have  $A(G)= \bigcap_{K\in \mathcal A^*(G)\cup \{G\}} K$.
\end{lemma}

\begin{proof} We may assume that $A(G)=1$ and that $G\ne 1$. Thus $G$ is almost semisimple by \cref{lem:almost residual}.  Let $N$ be a minimal normal subgroup of $G$. Then $N$ is a non-abelian simple group and $G/C_G(N)$ is almost simple with socle $N C_G(N)/C_G(N)$. Hence $C_G(N)\in \mathcal A^*(G)$. Therefore $$L =\bigcap_{ N \text { minimal normal  in }G}C_G(N)$$ is normal in $G$ and centralizes $\Soc(G)$. Hence $L=1$. This proves the claim.
\end{proof}

The next lemma illustrates how we use the Schreier property.

\begin{lemma}\label{lem:schreier app}
Suppose that $L \in \mathcal A^*(G)$. Let $K>L$ be such that $K$ is normal in $G$. Then $G/K$ is soluble. In particular, if $K,L\in\mathcal A^*(G)$ and $K\le L$, then $K=L$.
\end{lemma}

\begin{proof} Since $G/L$ is almost simple, the Schreier property yields $G/S$ is soluble for $S\ge L$ with $S/L=\Soc(G/L)$. Since $K>L$, we have $K\ge S$ and therefore $G/K$ is soluble. Now suppose that $K,L\in\mathcal A^*(G)$ and $K< L$. Then $G/L$ is soluble; however it is almost simple.
\end{proof}

\begin{lemma}\label{prop: G/A(G)}   For $J \in \mathcal A^*(G)$, define $I_J= \bigcap_{K\in \mathcal A^*(G)\setminus \{J\}}K$. Then $I_J/A(G)$ is isomorphic to  a non-trivial normal subgroup of the almost simple group $G/J$ and $$\Soc(G/A(G))= \prod_{J\in \mathcal A^*(G)} \Soc(I_J/A(G)) \cong \prod_{J\in \mathcal A^*(G)}\Soc (G/J).$$
\end{lemma}

\begin{proof} Plainly $I_J\ge A(G)$ and so, for this proof, we may assume $A(G)=1$.

Let us first demonstrate that $I_J$ is not contained in $J$. Suppose that this is false. Then $I_J \le J$. Choose $\mathcal B \subseteq\mathcal A^*(G)\setminus\{J\}$ minimal such that $ \bigcap_{K\in \mathcal B}K\le J$. If $|\mathcal B| =1$, then $\mathcal B=\{K\}$ and $K<J$ as $J \ne K$, contrary to \cref{lem:schreier app}. Hence $|\mathcal B| \ge 2$. Pick $\mathcal B_1$ and $\mathcal B_2$ distinct maximal subsets of $\mathcal B$. Then, for $i=1, 2$, set $U_i=  \bigcap_{K\in \mathcal B_i}K$. Of course, $U_1$ and $U_2$ are normal subgroups of $G$ with  $U_1\not \le J$ and $U_2\not \le J$ by the minimal choice of $\mathcal B$. Define $S>J$ to be the preimage of $\Soc(G/J)$. Then $S/J$ is a non-abelian simple group and  $S/J$ is a subgroup of  $U_1J/J$ and $U_2J/J$.  Now, using $S/J$ is a non-abelian simple group and $U_1\cap U_2=\bigcap_{K\in \mathcal B}K\le J$, we take commutators  \begin{eqnarray*}J&<&S =J[S,S]\le J[U_1J,U_2J]\le  J[J,U_1][J,U_2][U_1,U_2]\\&=& J[U_1,U_2] \le J(U_1\cap U_2)\le J,\end{eqnarray*} a contradiction. Hence $I_J \not \le J$. In particular, as $I_J \cap J=1$, $I_J\cong I_JJ/J$ which is a non-trivial normal subgroup of $G/J$.

Take $J\in \mathcal A^*(G)$. Since $\gen{I_L\mid L\in \mathcal A^*(G)\setminus\{J\}} \le J$ and $I_J \cap J=  1$,
we have $$\gen{I_J\mid J\in \mathcal A^*(G)} =\prod_{J\in \mathcal A^*(G)}I_J.$$

As   $I_J$ is isomorphic to a non-trivial normal subgroup of $G/J$ and $\Soc(G/J)$ is the unique minimal normal subgroup of $G/J$, we have $\Soc(I_J) \cong \Soc(I_JJ/J)= \Soc(G/J)$. Therefore $\Soc(I_J) \cong \Soc(G/J)$ and  $$\Soc(\prod_{J\in \mathcal A^*(G)}I_J) =\prod_{J\in \mathcal A^*(G)}\Soc(I_J)\cong \prod_{J\in \mathcal A ^*(G)}\Soc(G/J).$$

Assume that $N$ is a minimal normal subgroup of $G$. Then there exists $J\in \mathcal A^*(G)$ such that $N \not \le J$. Since $N$ is a minimal normal subgroup, $N\cap J=1$ and so $N$ centralizes $J$ and $N \cong NJ/J= \Soc(G/J)$ is non-abelian. Suppose that $K\in \mathcal A^*(G)\setminus \{J\}$ is such that $N \not \le K$. Then $N$ and $K$ commute. Thus $N\cong NJ/J$ commutes with $JK/J$. Therefore $K \le J$, and so $K=J$ by \cref{lem:schreier app}. Hence $N \le K$ and consequently $N\le I_J $. Therefore $N \le \Soc(\prod_{J\in \mathcal A^*(G)}I_J)$ and we conclude that $\Soc(G)=\Soc(\prod_{J\in \mathcal A^*(G)}I_J)$. This completes the proof.
\end{proof}

\begin{lemma}\label{lem:normalizes sylow}
Suppose that $n \ge 2$ is a natural number, and that $N= K_1\times\dots \times K_n$ is a non-abelian minimal normal subgroup of the group $G$. Assume that $\gen{x}\le G$ acts semiregularly on $\{K_1,\dots, K_n\}$ by conjugation. Then for all $r \in \pi(N)$, there exists $R \in \Syl_r(N)$ with $x \in N_G(R)$.
\end{lemma}

\begin{proof} Suppose that $G$ is a minimal counter example to the claim and let $r \in \pi(N)$.
Set  $H= N\gen{x}$. Then $N$ decomposes as a direct product $N_1\times \dots \times N_k$ of minimal normal subgroup of $H$.   Since all the components of $N$ are isomorphic, $r \in \pi(N_i)$ for all $1\le i \le k$.

If $H<G$, then by induction there exists $R_i\in \Syl_r(N_i)$ with $x \in N_H(R_i)$ for $1\le i \le k$. But then $x \in N_G(\prod_{i=1}^kR_i )$ and $\prod_{i=1}^kR_i \in \Syl_r(N)$, which is a contradiction.  Hence $G=H$ and, as $N$ is a minimal normal subgroup of $G$,  $\gen{x}$ acts regularly  on $\{K_1,\dots, K_n\}$ by conjugation and so we can choose notation so that $K_{i+1}=K_1^{x^i}$, $0\le i\le n-1$.
Let $R_1\in \Syl_r(K_1)$. Then $R_{i+1}= R_1^{x^i}\in \syl_r(K_1^{x^i})=\Syl_r(K_{i+1})$. Now, as $\gen{x}$ acts regularly on $\{K_1,\dots, K_n\}$,  $\gen{R_1^{\gen{x}}}= R_1  \dots  R_n\in \Syl_r(N)$ and   $x \in N_G(R)$ where $R=R_1  \dots  R_n$, a contradiction.
\end{proof}

\begin{lemma}\label{lem:structure}
Suppose that $s$ and $t $ are distinct primes with and $G$ is a   group with exactly one minimal normal subgroup $N$. If $G$ is   invariably $(s,t)$-generated, then either $N$ is abelian or $G$ is almost simple with $\{s,t\}\subseteq \pi(|N|)$.
\end{lemma}

\begin{proof} In this proof we do not assume $s<t$.
Suppose that $N$ is not abelian. Then, with $n \ge 1$, $$N=K_1\times \dots \times K_n$$  with each $K_i$ a non-abelian simple group. Set $\mathcal K=\{K_1,\dots, K_n\}$.  Let $(\sigma, \tau)$ be an invariable $(s,t)$-sequence for $G$. Assume that $|N|$ is coprime to $st$. Let $P \in \syl_p(N)$ for some $p \in \pi(|N|)$.  Then $G= N_G(P)N$ by the Frattini argument. Since $|N|$ is coprime to $st$, we have $N_G(P)$ contains both a Sylow $s$-subgroup and a Sylow $t$-subgroup of $G$. Hence there exist $g,h \in G$ such that $G>N_G(P) \ge \gen{\sigma^g,\tau^h}= G$, a contradiction. Therefore $|N|$ and $st$ are not coprime. Without loss of generality, we may suppose that $s \in \pi(|N|)$. Choose $S \in \Syl_s(G)$ with $\sigma \in S$.
 Assume that $t \not \in \pi(|N|)$. Then $N_G(S\cap N)$ contains a Sylow $t$-subgroup of $G$ by the Frattini argument. Hence $\tau$ can be replaced by a conjugate in $N_G(S\cap N)$. But then $\gen{\sigma,\tau} \le N_G(S\cap N)<G$, a contradiction.
 Thus $t \in \pi(|N|)$ and we conclude  that $\{s,t\} \subseteq \pi(|N|)$.

If $\sigma$ acts semiregularly by conjugation on $\mathcal K$, then $\sigma$ normalizes a Sylow $t$-subgroup of $N$ by \cref{lem:normalizes sylow}. Therefore we may replace $\sigma$ by a suitable conjugate  so that $\sigma$   normalizes $T\cap N$ where $T$ is a Sylow $t$-subgroup which contains $\tau$. But then $G =\gen{\sigma,\tau} \le N_G(T\cap N)$, a contradiction.  Similarly, $\tau$ does not act semiregularly on $\mathcal K$. As $\sigma$ and $\tau$ have prime order, $\sigma$ and $\tau$ both normalize members of $\mathcal K$. Since $G$ acts transitively on $\mathcal K$ by conjugation, we may   conjugate $\sigma$ and $\tau$ independently by $g$ and $h\in G$ so that the conjugates normalize $K_1$. But then $K_1 $ is normalized by $\gen{\sigma^g,\tau^h}=G$ and therefore $N=K_1$.  As $C_G(N)=1$, we obtain $G$ is almost simple.
\end{proof}

 In the next lemma, for $X$ a group, $\mathrm{Sol}(X)$ denotes the largest normal soluble subgroup of $X$. We have $\mathrm{Sol}(X)$ is the product of all the soluble normal subgroups of $X$. We also use the fact that for $L$ normal in $X$   with  $L \le A(X)$, we have $A(X/L)= A(X)/L$.

\begin{lemma}\label{lem:AGsol}
Suppose that $s$ and $t $ are distinct primes. If $G$ is   invariably $(s,t)$-generated, then  $A(G)$ is soluble.
\end{lemma}

\begin{proof}
Suppose that $G$ is a minimal counterexample to the claim. Then $A(G)$ is not soluble. If $\Sol(A(G)) \ne 1$, then $A(G)/\Sol(A(G)) =A(G/\Sol(A(G)))$ is not soluble and $G/\Sol(A(G)) $ is $(s,t)$-invariably generated by \cref{lem:quotients}. The minimality of $G$ yields $\Sol(A(G))=1$. Let $K$ be a minimal normal subgroup of $G$ in  $A(G)$ and set $L= C_{G}(K)$. Then, as $\Sol(A(G))=1$, $K$ is non-abelian and $K \cap L=1$. In particular, $KL/L\cong K$ is a minimal normal subgroup of $G/L$. Assume $L_1$ is the preimage in $G$ of $C_{G/L}(KL/L)$. Then $[L_1,K,K]\le [L,K]=1$ and the fact that $K$ is perfect combined with the Three Subgroups lemma yields $L_1=L$. Hence $KL/L$ is the unique minimal normal subgroup of $G/L$.
Now \cref{lem:structure} yields $G/L$ is almost simple. Hence $L \in  \mathcal A^*(G)$. Therefore $L \ge A(G)\ge K$ and $L$ centralizes $K$, a contradiction as $K$ is non-abelian. We conclude that $A(G)$ is soluble.
\end{proof}

By \cref{lem:AGsol}, we know $A(G)$ has   Hall  subgroups.

\begin{lemma}\label{lem:A(G) structure}
Suppose that $s$ and $t$ are distinct primes  and $G$ is an invariably  $(s,t)$-generated group. Let $H$ be a Hall $\{s,t\}$-subgroup of $A(G)$.  Then $H$ is normal in $A(G)$, $A(G)/H$ is nilpotent and $\pi(|A(G)/H|)\subseteq \pi(|G/A(G)|).$
\end{lemma}

\begin{proof} Let $(\sigma, \tau)$ be an invariable $(s,t)$-sequence for $G$. We may as well suppose that  $G \ne A(G)$ for otherwise $H=G$ and the statement holds.

By the Frattini argument $G= N_G(H)A(G)$. Hence $|G:N_G(H)|=|A(G):N_{A(G)}(H)|$ is coprime to $st$. In particular, $N_G(H)$ contains both a Sylow $s$-subgroup and a Sylow $t$-subgroup of $G$. Therefore $\sigma$ and $\tau$ are conjugate   into $N_G(H)$ and consequently $G= N_G(H)$. That is $H$ is normal in $G$.

Assume that $r  \in \pi(|A(G)|)\setminus \{s,t\}$ and let $R$ be a Hall $\{s,t,r\}$-subgroup of $A(G)$. Then again, the Frattini argument implies that $R$ is normal in $G$. Hence $R/H= O_r(A(G)/H)$. We conclude that $A(G)/H$ is nilpotent. Assume that $R> H$. We claim that $r \in \pi(|G/A(G)|)$. Suppose that this is false and set $L= O^r(A(G))$. Then $L< A(G)$, and so $G/L$ is invariably $\{s,t\}$-generated by \cref{lem:quotients}. To simplify notation we assume that $L=1$. Then $|A(G)|$  is coprime to $|G/A(G)|$. Therefore, the Schur-Zassenhaus theorem implies that $G$ has a complement $M$ to $A(G)$. Clearly, $M$ contains both a Sylow $s$-subgroup and a Sylow $t$-subgroup of $G$ and so $G$ is not invariably $\{s,t\}$-generated, a contradiction. Hence $r \in \pi(G/A(G))$.
\end{proof}

\begin{proof}[The proof of Theorem A] Suppose that $G$ is as in the hypothesis of Theorem A.

If $G=A(G)$, then $G$ is an $\{s,t\}$-group by \cref{lem:Sylow/Hall} (ii). Suppose that $G$ is not soluble. Then $ K_2=A(G)$, is soluble because of  \cref{lem:AGsol}. Therefore   $G/K_2$ is not soluble and $K_2<G$. We have that $G/K_2$ is almost semisimple by \cref{prop: G/A(G)}.
Define $K_1$ to be the preimage of $\Soc(G/K_2)$. Let $J\in \mathcal A^*(G)$. Then $G/J$ is invariably $(s,t)$-generated  by \cref{lem:quotients} and $G/J$ has a unique minimal normal subgroup $\Soc(G/J)$ and so \cref{lem:structure} implies $\{s,t\} \subseteq \pi(|\Soc(G/J))|$. This proves Theorem A (ii).
Since $C_{G/K_2}(\Soc(G/K_2))=1$, $G/K_1$ is a soluble group by the Schreier property of finite simple groups. As $G/K_1$ is invariably $(s,t)$-generated, we deduce that $G/K_1$ is an $\{s,t\}$-group. This is Theorem A (i).  The structure of $K_2$ as given in Theorem A (iii) and (iv) is   described by \cref{lem:A(G) structure} once we add the additional observation that $G/K_1$ is an $\{s,t\}$-group and $\pi(|K_2/O_{s,t}(K_2)|)\subseteq \pi(|G/K_2|)$ with $\{s,t\} \cap \pi(|A(G)/O_{s,t}(A(G))|)=\emptyset$.
\end{proof}

We now prove Theorem B.

\begin{proof}[Proof of Theorem B] Suppose that $T \le G$ where $G$ is invariably $(s,t)$-generated and $T$ is a non-abelian simple group. Then, as $A(G)$ is soluble, $T\not \le A(G)$.  Hence, by \cref{lem:AG}, there exists $J \in \mathcal A^*(G)$ such that $T \not \le J$.  Since $J$ is normal in $G$ and $T$ is a simple group, we have $J \cap T=1$. Thus $T \cong TJ/J \le G/J$ and $G/J$ is an almost simple group which is invariably $(s,t)$-generated. This proves the claim.
 \end{proof}

\begin{proposition}\label{lem:dif socles} Suppose that $\Lambda$ is a  sequence of positive integers, $A_1, \dots, A_k$, $k\ge 1$, are   groups which are invariably $\Lambda$-generated. Set $H=A_1\times \dots\times A_k$.
\begin{enumerate}
\item $H$ has  a subgroup $G$ which is invariably $\Lambda$-generated and projects onto every $A_i$, $1\le i \le k$.
\item If, in addition, each $A_i$ is almost simple, and, for $1\le i<j\le k$, $ \Soc(A_i) \not \cong \Soc(A_j)$, then $G$ in part (i) is almost semisimple and  $\Soc(G)=\Soc(H)$.
\end{enumerate}
\end{proposition}

\begin{proof} Let $\Lambda=(\lambda_1, \dots,\lambda_n)$.  Assume that, for $1\le i\le k$, $A_i$ is invariably $\Lambda$-generated by the invariable $\Lambda$-sequence $(s_{i,1}, \dots, s_{i,n})$.

For $1\le j \le n$, define $$\sigma_j=(s_{1,j}, \dots, s_{k,j}) \in H.$$ Then, as $s_{i,j}$ has order dividing $\lambda_j$,   $\sigma_j$ has order dividing $\lambda_j$.

Let $$ \Sigma = \{(\sigma_1^{x_1}, \dots,\sigma_n ^{x_n})\mid x_i=(x_{1,i}, \dots, x_{k,i}) \in H, 1\le i \le n\}\subset H^n.$$

Pick $\theta =(\theta_1, \dots, \theta_n)\in \Sigma$ such that $G=\gen{\theta_1,\dots,\theta_n}$ has minimal order and write $$\theta_j=  (s_{1,j}, \dots, s_{k,j})^{x_j}=(s_{1,j}^{x_{1,j}},\dots,s_{k,j}^{x_{k,j}})$$ for $1\le j \le n$.

Suppose that $g_1,\dots,g_n \in G$. Then $(\theta_1^{g_1},\dots, \theta_n^{g_n})\in \Sigma$. The minimal choice of $G$ yields $G=\gen{\theta_1^{g_1},\dots, \theta_n^{g_n}}$ and thus $G$ is invariably $\Lambda$-generated.
 Denote by $\pi_i$ the projection of $G$ onto the $i$th factor of $H=A_1 \times \dots \times A_k$. As    $(s_{i,1}, \dots, s_{i,n})$  is an invariable $\Lambda$-sequence for $A_i$,    $$\pi_i(G)=\gen{\pi_i(\theta_1), \dots, \pi_i(\theta_n)} =\gen{s_{i,1}^{x_{i,1}},\dots, s_{i,n}^{x_{i,n}}}=A_i.$$  This completes the proof of (i).

 We now suppose that the hypothesis of (ii) holds and show that $G$ is almost semisimple   and $\Soc(G)=\Soc(H)$.
Let $P_i=\ker \pi_i$. Then, as $\pi_i(G)=A_i$ by (i), $G/P_i\cong A_i$. Hence $G/P_i$ is almost simple and $P_i\in \mathcal A^*(G)$. Since $\bigcap_{i=1}^kP_i=\bigcap_{i=1}^k\ker\pi_i=1$, we have $A(G)=1$ and so $G$ is almost semisimple by \cref{lem:almost residual}.
 Suppose that $1\le i <j \le k$ and $P_i=P_j$.  Then $$\Soc(A_i)\cong \Soc(G/P_i)= \Soc(G/P_j)\cong \Soc(A_j),$$ which contradicts our assumption that $\Soc(A_i)\not\cong \Soc(A_j)$. Hence $P_i \ne P_j$.
Invoking \cref{prop: G/A(G)} shows that $\Soc(G)$ has a subgroup $X$ with $$X\cong \prod_{i=1}^k \Soc(G/P_i)\cong \prod_{i=1}^k \Soc(A_i).$$ Since, by the Schreier property, $H/\Soc(H)$ is soluble and $X$ is perfect, $X\le \Soc(H)$. Now $|X| = |\Soc(H)|$ yields $\Soc(H)=X$. As $G$ is almost semisimple, every minimal normal subgroup of $G$ is a non-abelian simple group, and as $H/\Soc(H)$ is soluble, we have $\Soc(H)=X \le \Soc(G) \le \Soc(H)$ and so $\Soc(H)=\Soc(G)$. This proves (ii).
\end{proof}

We close this section with two general results which do  not fit elsewhere.

\begin{lemma}\label{lem:contain r and s subgroups} Suppose that $r$ and $s$ are distinct  prime numbers and that $F$ is a group with a normal subgroup $K$ such that $F/K$ is soluble. Let $G$ be such that $K \le G \le F$. Assume that $H\le F$  contains a Sylow $r$-subgroup and a Sylow $s $-subgroup of $F$ and that $ K   \not \le H$. Then $G$ is not  $(r^{|G|},s^{|G|})$-invariably generated.
\end{lemma}

\begin{proof} Suppose that $(\ell,m)$ is an invariable $(r^{|G|},s^{|G|})$-sequence for $G$.

  Then $G/K$ is $(r^{|G|},s^{|G|})$-invariably generated by \cref{lem:quotients}. Since $F/K$ is soluble, $G/K$ is  an $\{r,s\}$-group by \cref{lem:Sylow/Hall}. Since $HK/K$ contains a Sylow $r$-subgroup and a Sylow $s$-subgroup of $F/K$ which is soluble, $HK/K$ contains a Hall $\{r,s\}$-subgroup of $F/K$  by Hall's Theorem. Thus we can conjugate $H$ in $F$ so that $HK/K \ge G/K$.  Hence $HK \ge G$ and $G= (G \cap H)K$. In particular, $G/K= (G \cap H)K/K\cong (G\cap H)/(G\cap H \cap K)=(G\cap H)/(H\cap  K)$ and, as $H \cap K$ contains both a Sylow $r$-subgroup and Sylow $s$-subgroup of $K$, we conclude that $G\cap H$ contains both a Sylow $r$-subgroup and a Sylow $s$-subgroup of $G$. Hence there are $f_1,f_2\in G$ such that $K\le G=\gen{\ell^{f_1}, m^{f_2}}\le G\cap H\le H$, which is a contradiction. Hence $G$ is not  $(r^{|G|},s^{|G|})$-invariably generated.
\end{proof}

\begin{lemma}\label{lem:lastlem}
Suppose that $s$ is a prime  and $t$ is a natural number, $G$ is almost simple with socle $K$ and    $x,y\in G$ with $x$ of order $s$ and $y$ of order $t$. Assume   $1\ne D < K$ and $\Aut(K) =N_{\Aut(K)}(D) K$.  Assume that $H= N_G(D)$  contains both a Sylow $s$-subgroup of $G$ and an $\Aut(K)$-conjugate of $y$.  Then $(x,y)$ is not an invariable $(s,t)$-sequence for $G$.
\end{lemma}

\begin{proof} Suppose that $(x,y)$ is an invariable $(s,t)$-sequence for $G$. Assume that $\alpha \in \Aut(K)$ with $y^\alpha \in H$.
  We have that $$G^\alpha = G^\alpha\cap \Aut(K)= G^\alpha\cap N_{\Aut(K)}(D) K=  N_{G^\alpha}(D)K$$ and so, as $N_{G^\alpha}(D)/N_K(D)\cong G^\alpha/K$ and $N_K(D)$ contains a Sylow $s$-subgroup of $K$, $N_{G^\alpha}(D)$ contains a Sylow $s$-subgroup of $G^\alpha$.
Now $y^\alpha$ normalizes $D$ and $x^\alpha$ is $G^\alpha$-conjugate, by $g^\alpha$ say, into a Sylow $s$-subgroup of $N_{G^\alpha}(D)$.
Hence $$\gen{(x^g)^\alpha,y^\alpha} \le N_{G^\alpha}(D) < G^\alpha,$$ a contradiction as $\gen{(x^g)^\alpha,y^\alpha}=  \gen{x^g ,y}^\alpha=G^\alpha$.
\end{proof}
\section{Broad   and $p$-broad subgroups}\label{sec:3}

In \cite{GuralnickRobinson1}, Guralnick and Robinson define an elementary abelian $2$-subgroup $E$ of a finite group $G$ to be \emph{broad} provided $z^G\cap E \not=\emptyset$ for all involutions $z\in G$. We generalize this concept.

 \begin{definition}\label{def:pbroad} For a prime $p$ and group $G$, we define a \emph{$p$-broad subgroup} to be an elementary abelian $p$-subgroup $E$ of $G$ such that $e^G \cap E\ne \emptyset$ for all $e \in G$ of order $p$.
\end{definition}

\begin{lemma}\label{lem:2 class broad}
Suppose that $p$ is a prime, $G$ is a group and $G$ has at most two conjugacy classes of elements of order $p$. Then $G$ has a $p$-broad subgroup.
\end{lemma}

\begin{proof} We may assume that $G$ has two conjugacy classes of elements of order $p$ with representatives $x$ and $y$.  Let $T \in \Syl_p(G)$   and assume that $x\in Z(T)$. Then there exists $g \in G$ such that $y^g \in T$. Thus $\gen{x,y^g}$ is $p$-broad.
\end{proof}

With this definition, a $2$-broad subgroup is a broad subgroup.

\begin{theorem}[Guralnick-Robinson]\label{GR} Suppose that $G$ is a quasisimple group. Then $G$ has a broad subgroup.
\end{theorem}
\begin{proof}  This is \cite[Theorem 1]{GuralnickRobinson1}.\end{proof}

We extend this result modestly as follows:
\begin{lemma}\label{lem:Broad}
Suppose that $G$ is an almost simple group with socle $K$ either a non-abelian simple alternating group or a sporadic simple group. Then $G$ has a broad subgroup or $G \cong \Aut(\Alt(6))$ and in this case $G$ does not have a broad subgroup.
\end{lemma}

\begin{proof}The result for groups with $K$ an alternating group other than $\Alt(6)$ is obvious (see \cite[Lemma 2.1]{GuralnickRobinson1}). For $K$ a sporadic simple group, using \cref{GR} we may assume that $G$ is not perfect. The   {\sc magma} \cite{Magma} code in the supplemental materials finds a broad subgroup in all cases except $\Aut(\mathrm{HN})$ and $\mathrm{Fi}_{24}$ (where we were not patient enough).

For $G=\mathrm{Fi}_{24}$, fortunately the centre of a Sylow $2$-subgroup of $2\times \mathrm{Fi}_{23}$ is a broad subgroup of $G$. For $\Aut(\mathrm{HN})$, we find an involution with centralizer $2\times \Sym(10)$ and quickly  check that this group contains a representative of every conjugacy class of involutions in $G$. As $2\times \Sym(10)$   has a broad subgroup of order $2^6$, we are done. Finally, we remark that for $G=\Aut(\Alt(6))$, there are three subgroups of index $2$ and the subgroup $\mathrm{Mat}(10)$ has no involutions in its outer half. Since both $\Sym(6) $ and $\PGL_2(9)$ do, it is impossible for $G$ to have a broad subgroup.
\end{proof}

We now start to explore $p$-broad subgroups, with $p$ odd.
The first result is visibly true.
\begin{lemma}\label{lem:pbroadalt} Suppose that $G$ is almost simple and $K=\Soc(G)$ is an alternating group of degree at least $5$. Then, for $p$ an odd prime, $G$ has a $p$-broad subgroup.
\end{lemma}

\begin{proof} The subgroup generated by $\lfloor \frac n p\rfloor $ $p$-cycles with pairwise different support contains a member of every conjugacy class of elements of order $p$.
\end{proof}

\begin{lemma}\label{lem:pbroadspor} Suppose that $G$ is almost simple with $K=\Soc(G)$  is a sporadic simple group. Then, for $p$ an odd prime, $G$ has a $p$-broad subgroup unless $(K,p) \in \{(\HS,5), (\He,7)\}$.
\end{lemma}

\begin{proof} Suppose that $K$ is a sporadic group and $p$ is an odd prime. If the Sylow $p$-subgroups are abelian, or if $K$ has at most $2$ conjugacy classes of subgroups of order $p$, then the result is true. Working through \cite[Table 5.3]{GLS3}, we need to consider the pairs $(K,p)$ one of
$ (\Co_3,3),(\Co_1,3),(\Co_1,5),(\HS,5),(\Suz,3),(\He,7),
(\Fi_{22},3),(\Fi_{23},3),\\(\Fi_{24},3),(\HN,5)$ or $(\M,3)$.

Let us consider  the largest case: $(K,p)=(\M,3)$. For this we resort to the determination of elementary abelian $3$-subgroups in $\M$ \cite[Table VI]{Richardson}. From there we see that the subgroup $B_6$ contains each conjugacy class of elements of order $3$. Hence $\M$ has a $3$-broad subgroup.
Consider next, the special cases of $K=\HS$ with $p=5$ or $K=\He$ with $p=7$. In these cases, by \cite[Tables 5.3m and 5.3p]{GLS3},  there are three conjugacy classes of cyclic subgroup of order $p$ and $K$ has extraspecial Sylow $p$-subgroups of order $p^3$ with representative $S$ say. Thus the maximal elementary abelian subgroups of $S$ have order $p^2$ and contain $p$ cyclic   subgroups which are not equal to the centre of $S$ all of which are conjugate in $S$. This proves the claim that $\HS$ has no $5$-broad subgroups and  $\He$ has no $7$-broad subgroups.

The remaining cases are shown to have a $p$-broad subgroup using a {\sc magma} computation using the code in the supplemental materials. Now $G$ inherits its $p$-broad subgroups from $K$ and in the cases where $K$ has no broad subgroups, the explanation of this for $K$ also works for $G$ as the conjugacy classes of cyclic subgroups of order $p$ remain distinct in $G$ \cite[Tables 5.3m and 5.3p]{GLS3}.
\end{proof}

We now start to investigate $p$-broad of simple groups of Lie type defined in characteristic $r$. If $p=r$ is odd, then, unlike for $p=2$, we expect that mostly there are no $p$-broad subgroups.

\begin{example}
The group $\PSL_4(p^a)$ has no $p$-broad subgroup for $p \ge 5$.
\end{example}

\begin{proof}
It suffices to calculate in $G=\SL_4(p^a)$. Let $V$ be the standard $4$-dimensional module for $G$,  $x= \left(\begin{smallmatrix} 1&0&0&0\\1&1&0&0\\0&1&1&0\\0&0&0&1\end{smallmatrix}\right)$
and
$y= \left(\begin{smallmatrix} 1&0&0&0\\1&1&0&0\\0&1&1&0\\0&0&1&1\end{smallmatrix}\right)$.
As $p \ge 5$, both $x$ and $y$ have order $p$. We claim that no conjugate of $x$ commutes with $y$. To see this assume that $x_0 \in x^G\cap C_G(y)$.  Then $[V,x_0]$ has dimension $2$ and $C_V(x_0)$ has dimension $2$. Since $y$ stabilizes  exactly one subspace of dimension $2$, we must have $C_V(x_0)=[V,x_0]$, a contradiction as $(x_0-1)^2 \ne 0$. Since no conjugate of $x$ commutes with $y$, $G$ cannot have a  $p$-broad subgroup as such a subgroup would contain a conjugate of $x$ and $y$.
\end{proof}

We have also checked, using {\sc magma} that $\PSU_4(3)$ has no $3$-broad subgroups.

Recall that for   a    group of Lie type a prime $p$ is   \emph{good}   provided   that   the coefficients in the decomposition of the highest root into a sum of fundamental roots only has coefficients strictly less than $p$ \cite[Theorem 4.10.3 (e)]{GLS3}.

\begin{lemma}\label{lem:p-good}
Suppose that $L$ is a universal group of Lie type defined in characteristic $r$ and assume that $p\ne r$ is an odd prime which is good for $L$. Assume that $|Z(L)|$ is coprime to $p$ and that $L\not\cong \SL_3^\pm(r^a)$. Then $L$ has a $p$-broad subgroup.
\end{lemma}

\begin{proof}Let $e \in L$ have order $p$. By \cite[Theorem 4.10.3(c)]{GLS3}, with the restrictions on $p$ as given, $L$ has a unique conjugacy class of   elementary abelian $p$-subgroups of maximal $p$-rank. Let  $A$ be such a subgroup.
Then     by \cite[Theorem 4.10.3(e)]{GLS3}, $\gen{e}$ is conjugate to a subgroup of $A$. Hence $e^L \cap A$ is non-empty and we are done.
\end{proof}

\begin{lemma}\label{lem:3good} For $r\ne 3$, the following groups have a $3$-broad subgroup:
\begin{enumerate}
\item $\PSL_n^\epsilon(r^a)$ when $n> 3$ and $(r^a-\eps,n)$ is coprime to $3$;
\item $\PSp_{2n}(r^a)$, $n \ge 1$, $\POmega_{2n-1}(r^a)$ and $\POmega_{2n}^\pm(r^a)$, $n \ge 4$.
\end{enumerate}
\end{lemma}

\begin{proof} For the groups in question,   $3$ is a good prime by \cite[Table 1.8]{GLS3}. Furthermore, in case (i) we have eliminated the possibility that $3$ divides the centre of the universal  groups $\SL_n(r^a)$ and $\SU_n(r^a)$. Hence in cases (i) and (ii),  \cref{lem:p-good} implies that the universal version $L$ of the groups in question have a $p$-broad subgroup $E$.  Since $Z(L)$ has order coprime to $p$, it follows that the image of $E$ in the quotient group is also $p$-broad.
\end{proof}

The next example demonstrates that the condition in \cref{lem:3good} (i) is required.

\begin{example}
The group $\PSL_9(4)$ has no $3$-broad subgroups.
\end{example}

\begin{proof} We sketch an explanation.
Let $X\cong \GL_9(4)$, $Y= \SL_9(4)$, $\omega \in \GF(4)$ have order $3$ and set  $$x=\diag(\left(\begin{smallmatrix} 0&1&0\\0&0&1\\\omega&0&0\end{smallmatrix}\right), \left(\begin{smallmatrix} 0&1&0\\0&0&1\\\omega&0&0\end{smallmatrix}\right),\left(\begin{smallmatrix} 0&1&0\\0&0&1\\\omega&0&0\end{smallmatrix}\right)).$$
Then $x\in Y$  has order $9$ and $\gen{x^3}= Z(Y)$.
We have $N_X(\gen{x}) \cong \GL_3(64)\gen{\phi}$ where $\phi$ is defined from the field automorphism of $\GF(64)$ fixing $\GF(4)$. Let $$y=\diag(\omega, \omega^{-1},1,1,1,1,1,1,1).$$ Then $y\in Y$ has order $3$ and is not conjugate into $C_X(x)$ as the dimension of the fixed space of $y$ is not a multiple of $3$. Suppose that $y$ is conjugate to an element $y_0$ of $N_X(\gen x)\setminus C_X(x)$. Then $y_0$ induces by conjugation a field automorphism on $\SL_3(64)$ and therefore commutes with a subgroup $M$ isomorphic to $\SL_3(4)$. The group $M$ decomposes the natural module for $Y$ as a direct sum of three irreducible $3$-dimensional modules. But then $y_0$ cannot commute with $M$, a contradiction. Hence $y$ is not conjugate to an element of $N_X(\gen{x})$. It follows that the image  of $x$ in $\PSL_9(4)$ does not commute with the image of any conjugate of $y$ (as such a conjugate would have preimage in $N_X (\gen{x})$). Hence $\PSL_9(4)$ has no $3$-broad subgroups.
\end{proof}

\begin{lemma}\label{3-broad} The groups ${}^2\F_4(2^a)'$, $\G_2(r^a)$, ${}^3\mathrm{D_4}(r^a)$, $\F_4(r^a)$, $\E_7(r^a)$ and $\E_8(r^a)$, $r \ne 3$, have $3$-broad subgroups.
\end{lemma}
\begin{proof}
If $K\cong {}^3\D_4(r^a), {}^2\F_4(2^a)'$ or $\G_2(r^a)$, using
\cite[Table 4.7.3A]{GLS3} with \cref{lem:2 class broad} shows that $K$ has a $3$-broad subgroup.

Let $K$ be one of the   groups  $\F_4(r^a)$, $\E_7(r^a)$ and $\E_8(r^a)$, with $r \ne 3$ and represent the Lie rank of $K$ by $\ell$. Assume that $\ov K$ is the algebraic group with $\sigma$-setup giving us $K$. Let $\ov T\le \ov B$ be a $\sigma$-stable torus in a $\sigma$-stable Borel subgroup.  Then $C_{\ov T}(\sigma)$ has order $\Phi_1(r^a)^\ell$.  Let $w_0 \in Z(N_{\ov K}(\ov T)/\ov T)$ have order $2$ (this exists because of the choice of $K$
in the first sentence). Let $g\in \ov K$ be such that $g\sigma(g^{-1}) \in w_0\ov T$. Then \cite[Propositions 25.1 and 25.3]{MalleTesterman} show that $ \ov T^g$ is $\sigma$-stable
and, letting $X$ be the character group of $\ov T$, $|C_{\ov T^g}(\sigma)|= |\det_{X\otimes \mathbb R}(w_0\sigma-1)|=\Phi_2(r^a)^\ell$ as $w_0$ inverts $X$ and $\sigma$ sends every element of $X$ to its $r^a$th power.

Since $3$ divides either $\Phi_1(r^a)$ or $\Phi_2(r^a)$, we have that $K$ has an elementary abelian subgroup $E$ of order $3^\ell$. We know that every $\ov K$-conjugacy class of elements of order $3$ is represented in $\ov T$. We also know the number of these classes by \cite[Table 4.7.1]{GLS3}. Thus $E$ witnesses at least this many $K$-conjugacy classes. On the other hand, \cite[Table 4.7.3A]{GLS3} says that there are exactly the same number of classes in the finite group. Thus $x^K \cap E \not=\emptyset$ for all $K$-conjugacy classes of elements of order $3$.
\end{proof}

\begin{lemma}\label{lem:E6 good cases} Suppose that $r\ne 3$ is a prime.
   If $(r^a-\eps,3)=1$, then $\E_6^\eps(r^a)$ has a $3$-broad  subgroup.
\end{lemma}

\begin{proof} Let $K\cong \E_6^{-\eps}(r^a)$ where $\eps\equiv r^a \pmod 3$. Then the subgroup $F\cong \F_4(r^a)$ of $K$ contains a Sylow $3$-subgroup of $K$ by \cite[Table 2.2]{GLS3}. Since $F$ has a $3$-broad subgroup by \cref{3-broad}, so does $K$.
\end{proof}

We close this section with a result which shows how $p$-broad subgroups will be used in our arguments.

\begin{lemma}\label{lem: broad subgroups}
Suppose that $s$ and $t$ are distinct primes, $G$ is  a group and  $(\sigma,\tau)$ is an invariable $(s,t)$-sequence for $G$. Assume $|Z(G)|$ is coprime to $s$.  If $G$ has an $s$-broad subgroup, then $|C_G(\tau)|$ is coprime to $s$.
\end{lemma}

\begin{proof}  For a contradiction, assume $w\in C_G(\tau)$ is an   element of order $s$. Then by hypothesis there exists an $s$-broad subgroup $A$ of $G$ which contains $w$. Since $\sigma ^G\cap A \not=\emptyset$, $C_G(w)$ contains $\tau$ and a conjugate of $\sigma$. As  $(\sigma,\tau)$ is an invariable $(s,t)$-sequence, we have $w \in Z(G)$ which contradicts $|Z(G)|$ being coprime to $s$.
\end{proof}

\section{Peripheral  results about groups of Lie type}\label{sec:4}

In this section we provide some facts about groups of Lie type that will be used in the rest of the paper.

\begin{proposition}\label{GLfact} Suppose that $q$ is odd, $\eps=\pm 1$  and $G= \GL_n^\epsilon(q)$. Set $$H= \GL_1^\epsilon(q) \wr \Sym(n).$$  Then the following hold.
\begin{enumerate}
\item If $q\equiv \epsilon \pmod 4$, then $H$ contains a Sylow $2$-subgroup of $G$.
\item If $q\equiv \epsilon \pmod 3$, then $H$ contains a Sylow $3$-subgroup of $G$.
\item The quotient $H/Z(G)$ contains a representative of every conjugacy class of involutions in $G/Z(G)$.
\end{enumerate}

\end{proposition}
\begin{proof} Parts (i), and (ii)  are elementary deductions from the   expressions for $|\GL_n^\eps(q)|$.

To prove (iii),  we use \cite[Table 4.5.1]{GLS3} for our notation for involutions in $\PGL_n(q)$ and let $\wt G$ denote $G/Z(G)$.  The involutions $t_i$ in $\wt G$, $1\le i \le \lfloor n/2\rfloor  $ have preimages  $s_i=\diag(-1,\dots,-1,1,\dots 1)$ with $i$ instances of $-1$. These involutions are all in $H$ and so $\wt H$ contains all the involutions $t_i$, $1 \le i \le \lfloor n/2\rfloor$.
When $n$ is even, there is an additional class of involutions  $t'_{n/2}$ in $\wt G$. These are not involutions in $G$.
Let $s\in H$ be the permutation matrix which corresponds to the permutation $(1,n/2+1) \dots (n/2,n)\in \Sym(n)$.
  Then $x=   t_{n/2}s\in H$ has order $4$ and squares to $t_n$. Hence $\wt x$  is conjugate to $t_{n/2}'$ and this proves (iii).
\end{proof}

\begin{lemma}\label{lem:E6 big 2-part} Suppose that $r \ge 5$, $K \cong \E_6^\eps(r^a)$ where $\eps=\pm 1$. Assume that $3$ divides $r^a-\eps$. Let $\gamma^*$ be a graph automorphism of  $K$ of order $2$. Then $|K:C_K(\gamma^*)|_2= (q-\eps)^2_2$.  In addition, for
 $x\in \Inndiag(K)$ of order $3$, $|C_K(x)|_2> (q-\eps)^2_2$.
\end{lemma}

\begin{proof} Let $K=\E_6(r^a)$. By \cite[Table 4.5.1]{GLS3}, $C_K(\gamma^*) \cong \F_4(r^a)$ or $\Inndiag(\PSp_8(r^a))$. We now check that $|K:C_K(\gamma^*)|_2= (q-\eps)^2_2$ using \cite[Table 2.2]{GLS3}. Then we  use \cite[Table 4.7.3A]{GLS3} to estimate the $2$-part of the centralizers $C_K(x)$. This gives the result.
\end{proof}

\begin{lemma}\label{lem:normalise a parabolic}
Suppose that $K$ is a simple group of Lie type defined in characteristic $p$. If $\alpha \in \Aut(K)$ normalizes a non-trivial $p$-subgroup of $K$, then $\alpha$ normalizes a parabolic subgroup of $K$.
\end{lemma}

\begin{proof} See \cite[Lemma  4.8]{ParkerSaunders}. \end{proof}
 
In the next lemma we adopt the notation from \cite[Definition 2.5.10, Theorem 2.5.12]{GLS3}. We assume that $\overline B$, $\overline T$, $\overline U$ come from a $\sigma$-setup for $K$ with corresponding subgroups $B$, $T$ and $U$ of $K$.

 \begin{lemma}\label{lem:factors}
Suppose that $K$ is a simple group of Lie type and $K=\Inn(K) \le G \le \Aut(K)$ is such that $G \le \Inndiag(K)\Phi_K$. Let $L$ be a parabolic subgroup of $K$, then $G= N_G(L)K$.
\end{lemma}

\begin{proof}  Set $G_0= \Inndiag(K)\Phi_K$.
We know that $\Phi_K$ normalizes all the parabolic subgroups of $K$ which contain $B$ and,   from \cite[Remark 2.5.11(d)]{GLS3} $\Inndiag(K)= C_{\overline T}(\sigma) K$. Thus, for $L$ a parabolic subgroup of $K$, we can conjugate $L$ to contain $B$ and then observe that $N_G(L) \ge C_{\overline T}(\sigma) \Phi_K$. Hence $$G_0\ge N_{G_0}(L)K \ge \Phi_KC_{\ov T}(\sigma)K= \Phi_K\Inndiag(K)= G_0$$ and so we have equality.

Now $G= G \cap G_0= G\cap N_{G_0}(L)K = (G \cap N_{G_0}(L))K= N_G(L)K$ as claimed.
\end{proof}

In the next lemma we study $K=\mathrm{P}\Omega_8^+(r^a)$. Let  $\Sigma$ be a root system for $K$ and label the Dynkin diagram so that the fundamental roots are $\{\alpha_1,\alpha_2,\alpha_3,\alpha_4\}$ with $\alpha_2$ being the middle node. Then the highest root is $\alpha^*=\alpha_1+2\alpha_2+\alpha_3+\alpha_4$.

\begin{lemma}\label{lem:NGD struct}
 Suppose that $r\ge 3$ is a prime, $q=r^a$,    $K\cong \mathrm P \Omega_{8}^+(q)$, and $K=\Inn(K)\le G\le \Aut(K)$.
   For $\alpha \in \Sigma$, define  $M_\alpha= \gen{X_{\alpha},X_{-\alpha}}$ and set $$D= M_{\alpha_1}M_{\alpha_3}M_{\alpha_4}M_{\alpha_*}.$$
 Then $G= N_G(D)K$, $N_G(D)$ contains a Sylow $2$-subgroup of $G$ as well as representatives of every $\Aut(K)$-conjugacy class of elements of order $3$ in $G \setminus K$.
\end{lemma}
\begin{proof} Notice that $ D$ is a commuting product of four fundamental $\SL_2(q)$-subgroups by \cite[Theorem 4.10.6]{GLS3}.
By \cite[Theorem 4.10.6(a),(b)]{GLS3}, $N_K(D)$ contains a Sylow
$2$-subgroup of $K$ and $D$ is determined up to $K$-conjugacy. The
Frattini argument yields $G=N_G(D)K$, and hence $N_G(D)$ contains a
Sylow $2$-subgroup of $G$.

Let $\tau$ be the standard triality automorphism defined with respect to $\Sigma$ as
in \cite[Theorem 2.5.1(d)]{GLS3}. Then $\tau$ permutes
$\{M_{\alpha_1},M_{\alpha_3},M_{\alpha_4}\}$ transitively and leaves
$M_{\alpha^*}$ invariant; moreover $x_{\alpha^*}(t)^\tau
=x_{\alpha^*}(\pm t)$, and $\tau^3=1$ forces
  $\tau$ to centralise $M_{\alpha^*}$. Since the standard
field automorphisms also normalize $D$, we have $\gen{\tau,\Phi_K}\le
N_G(D)$, and so $N_G(D)$ contains representatives of the classes of
field and of graph-field automorphisms of order $3$
\cite[Proposition 4.9.1]{GLS3}.

If $r=3$, then
\cite[Proposition 4.9.2(g)]{GLS3} shows that every element of order $3$ in the coset $\tau K$ is $\Inndiag(K)$-conjugate to an element of $\tau
X_{-\alpha^*}\subseteq \tau M_{\alpha^*} \subseteq N_G(D)$. So we may assume $r \ge 5$.

By \cite[Table 4.7.3A]{GLS3}, the coset $\tau K$ contains exactly two
$\Inndiag(K)$-classes of elements of order $3$, one represented by
$\tau$; let $\tau^*$ represent the other.    Choose $d\in M_{\alpha^*}$ of order $3$ and set
$\tau^\circ=\tau d$. Then $\tau^\circ$ has order $3$, lies in $\tau K$
and normalizes $D$. We claim $\tau^\circ$ is not conjugate to $\tau$,
so that $\tau^\circ$ is conjugate to $\tau^*$, which completes the
proof. Let $z$ be the unique involution in $Z(D)$ (note $\gen{\diag(-I_4,I_4),\diag(I_4,-I_4)}=Z(
\Omega_4^+(q)\times \Omega_4^+(q))$, so $Z(D)$ is the image in $K$ of a fours group
containing $-I_8$). It follows that $z$ is conjugate to $t_2$ in $K$.  Since $C_K(z)/D$ is a $2$-group
\cite[Table 4.5.1]{GLS3}, the components of $C_K(z)\cap
C_K(\tau^\circ)$ are visible in $D$. Indeed, As $\tau^\circ$ permutes
$M_{\alpha_1},M_{\alpha_3},M_{\alpha_4}$ cyclically and induces $d$ on
$M_{\alpha^*}$, there is exactly one, a diagonal $\SL_2(q)$. But if
$\tau^\circ$ were conjugate to $\tau$, then $C_K(\tau^\circ)\cong
\G_2(q)$ would contain the involution $z$, and involution centralizers
in $\G_2(q)$ have two $\SL_2(q)$-components \cite[Table 4.5.1]{GLS3},
a contradiction.
\end{proof}

 \begin{proposition}\label{lem:cop3}
 Suppose that $K$ is a non-abelian simple group  and $\alpha\in \Aut(K)$ has order $2$. Then either $|C_K(\alpha)|$ is divisible by $3$ or  $K$ and $\alpha$ are as follows:
 \begin{enumerate}
\item $K\cong \PSL_2(r^a)$, $r^a$ odd and either
 \begin{enumerate} \item $r=3$ and   $\alpha \in \Inndiag(K)$; or
 \item $r > 3$, $\alpha \in \Inndiag(K)$ with  $\alpha \in K$ if and only if $r^a\equiv \pm 5 \pmod {12}$.
 \end{enumerate}
 \item $\PSL_2(2^a)$, $\alpha \in K$.
 \item ${}^2\B_2(2^{2a+1})$, $\alpha \in K$.
 \item $\PSL_3^\eps(2^a)$, $\eps=\pm 1$  and $2^a-\eps$   not divisible by $9$, $\alpha\in K$.
\item $\PSp_4(2 ^a)'$,   $\alpha$ is a  graph-field automorphism or $\alpha \in K$ is in class $c_2$ (in the notation from \cite[(7.6)]{AschbacherSeitz}).
\item ${}^2\F_4(2^{2a+1})'$, $\alpha \in K$ a $2$-central element.
\item $\M_{22}$,   $\alpha$ is in class $2C$.
\item $\mathrm{Ru}$,  $\alpha$ is in class $2B$.
 \end{enumerate}
 \end{proposition}

 \begin{proof} Throughout this proof we use the fact that the only non-abelian quasisimple groups which have order not divisible by $3$ are covers of Suzuki groups ${}^2\B_2(2^{2a+1})$. We assume that $\alpha$ has order $2$ and $|C_K(\alpha)|$ is coprime to $3$.

  Suppose first that $K$ is a sporadic group. Inspecting \cite[Tables 5.3]{GLS3}, we obtain that $(K,\alpha)$ is  $(\M_{22}, 2C)$ or  $(\mathrm{Ru}, 2B)$.
  If $K$ is $\Alt(n)$ with $n\geq 5$, then  $K\cong \Alt(5)$ or $\Alt(6)$ and these cases are covered under the isomorphisms with $\PSL_2(5)$ and $\PSL_2(9)$ respectively.

Finally,  let  $K$ be a group of Lie type in characteristic $r$. Assume first that $r$ is odd.  Using \cite[Proposition 4.9.1]{GLS3} to determine centralizers of field and graph-field automorphisms, we have  $\alpha\in\Inndiag(K)\Gamma_K$, and now   \cite[Table 4.5.1]{GLS3} implies that $K\cong \PSL_2(r^a)$ and gives the further information in (i).

Suppose now that $r=2$. If $\alpha$ is either a field or a graph-field automorphism,    \cite[Proposition 4.9.1]{GLS3} eliminates all cases other than $K\cong \PSp_4(2^{2a+1})$ with $a\geq 1$ and $\alpha$ a graph-field automorphism. If $\alpha $ is a graph automorphism, \cite[(19.8) (i) and (19.9)]{AschbacherSeitz} show that $3$ divides $|C_K(\alpha)|$ unless perhaps
 $K\cong \Omega_{2n}^\eps(2^a)$. In this case, \cite[(7.6) and (8.10)]{AschbacherSeitz} says $\alpha$ acts as a $b_\ell$ type involution (with their notation) and \cite[(8.7)]{AschbacherSeitz} shows that such involutions have centralizer of order divisible by $3$.

 Thus, as $r=2$, $\alpha \in K$. Now  \cite[(4.3) and (6.2)]{AschbacherSeitz} shows that for $K\cong \PSL_n^\eps(2^a)$ we have $n \le 3$. Furthermore, if $n=2$ there are no restrictions.  If $n=3$, then for $|C_K(\alpha)|$ is coprime to $3$ if and only if $9$ does not divide $2^a-\epsilon$.

For $K$ a symplectic group, we use \cite[(7.6),(7.9),(7.10) and (7.11)]{AschbacherSeitz} to deduce that $K \cong \Sp_4(2^a)$ and $\alpha$ is of type $c_2$ (in the notation of \cite{AschbacherSeitz}).

For  $K$ an orthogonal group in dimension at least $8$, we need to consider involutions of type $a_\ell$ and $c_\ell$. For these cases we use \cite[(8.6) and (8.8)]{AschbacherSeitz} to see that no examples occur.

Suppose now that $K$ is an exceptional group. Here the centralizers are given in \cite[(13.3),(14.3),(15.6),
(16.20),(17.15),(18.4),(18.5),(18.6)]{AschbacherSeitz} and these yield only the examples   ${}^2\F_4(2^{2a+1})'$ with $\alpha$ $2$-central.\end{proof}

\section{Almost simple groups: alternating and sporadic simple groups}\label{sec:5}

We start with a general result about invariable $(s,t)$-generation of alternating groups.

 \begin{proposition}\label{lem:alt st}
 Suppose that $s$ and $t$ are primes with $s<t$ and that $n \ge st+1$.  Then $X=\Sym(n)$ and $Y=\Alt(n)$ are not invariably $(s,t)$-generated.
 \end{proposition}

\begin{proof}    We consider $Y \le G\le X$ acting naturally on $\Omega$ of size $n$. Let $(\sigma,\tau)$ be an invariable $(s,t)$-sequence for $G$.

First assume that $n>st$. Notice that both $\sigma$ and $\tau$    stabilize a subset   of $\Omega$ of size  $st$.  Since     $G$  acts transitively on subsets of size $st$, we may conjugate $\tau$ so that $\sigma$ and $\tau$ stabilize the same set of size $st$. Thus $\gen{\sigma,\tau}< G$.  Hence $n \le st$.
\end{proof}

\begin{remark}
The bound in \cref{lem:alt st} can be improved for $n \ge 5$ to $n> \max\{s(t-2),(s-1)t\}$.
\end{remark}

The objective of the next three sections is to prove the following theorem.

 \begin{theorem}\label{thm:Lietype23as} Suppose that $K$ is a non-abelian simple group  and $K=\Inn(K)\le G \le \Aut(K)$.
 Then $G$ is invariably $(2,3)$-generated if and only if
   $G  \cong \PGL_2(3^{2^b})$ for some $b\ge 1$.
 \end{theorem}

Until the proof of \cref{thm:Lietype23as} is complete, we assume that $K$ is a non-abelian simple group and $G$ is such that
 $K=\Inn(K)\le G \le \Aut(K)$  is invariably $(2,3)$-generated. Let $(z,d)$ be an invariable $(2,3)$-sequence for $G$.

In this section short section, we show that $K$ in \cref{thm:Lietype23as} cannot be a simple alternating group of degree $n \ne 6$ or a sporadic simple group.

\begin{lemma}\label{lem:alt gone}  If $K \cong \Alt(n)$ with $n \ge 5$, then $n=6$.
\end{lemma}

\begin{proof} By \cref{lem:alt st}, we have $n\le 6$. For $n=5$, we check that $\Sym(5)$ and $\Alt(5)$ are not invariably $(2,3)$-generated.
\end{proof}

\begin{lemma}\label{lem:sporadic gone} We have  $K$ is not a sporadic simple group.
\end{lemma}

 \begin{proof}   By \cref{lem:Broad} and \cref{lem:pbroadspor}, $G$ has a broad subgroup and a $3$-broad subgroup. By \cref{lem: broad subgroups}, $|C_G(z)|$ is coprime to $3$ and $|C_G(d)|$ is odd.
Using \cref{lem:cop3} combined with \cite[Tables 5.3]{GLS3}, we conclude that $K$  is not a sporadic group as in $\M_{22}$ and $\mathrm{Ru}$ no element of order $3$ has odd order centralizer.
 \end{proof}

\section{The groups $\mathrm{PSL}_2(r^a)$}\label{sec:6}

We continue towards the proof of \cref{thm:Lietype23as}, picking up the notation from the previous section.
Our goal here is to prove the result for $G$ with socle $K\cong \PSL_2(r^a)$. Remember that $(z,d)$ is an invariable $(2,3)$-sequence for $G$.

\begin{proposition}\label{lem:psl2}
 Suppose that $r$ is a prime, $K \cong \PSL_2(r^a)$. Then $G$ is invariably $(2,3)$-generated if and only if $G \cong \PGL_2(3^{2^b})$ for some $b\ge 1$.  \end{proposition}

\begin{proof}    Notice that $K$ has   exactly one conjugacy class of involutions.

\begin{claim}\label{PSL2clm2}  $d \in  K$.
\end{claim}

Suppose that $d \not \in K$. Then $d$ acts as a field automorphism on $K$ by \cite[Proposition 4.9.1]{GLS3}. Hence $d$ centralizes an involution in $K$. Since $z$ also centralizes an involution in $K$ and $K$ has a unique conjugacy class of involutions, this cannot happen. Hence \ref{PSL2clm2} holds.

\begin{claim} \label{PSL2clm1} $z\not \in K$.
\end{claim}

Suppose false.  Then $G=K$. If $r\ne 3$, then the Sylow $3$-subgroup of $K$ is cyclic and inverted by an involution. Hence $ z$ and $d$ can be conjugated independently into the normalizer of a Sylow $3$-subgroup, a contradiction. So suppose that $r=3$.  If $a$ is odd, then $K$ has a   unique conjugacy class of cyclic subgroups of order $3$ and contains a subgroup $L$ isomorphic to $\Alt(4)$.  But then $z$ and $d$ are conjugate into $L$, a contradiction. Hence $a$ is even and the Borel subgroup of $K$ contains a conjugate of $z$ and conjugate  of $d$, again a contradiction. This proves \ref{PSL2clm1}.

\medskip

Suppose first that $z$ is a field automorphism. Then $C_K(z) $ contains a subgroup isomorphic to $\PSL_2(r)$. If $r \ne 3$, then $d$ is conjugate into $C_K(z)$, a contradiction. So $r=3$. In which case, $z$ normalizes a Borel subgroup of $K$ and this contains a conjugate of $d$. This proves that $z$ is not a field automorphism. Hence $z\in \Inndiag(K)=\PGL_2(r^a)$ and $G \cong \PGL_2(r^a)$ with $r$ is odd. Furthermore, $z$ is uniquely determined up to $G$-conjugacy by \cite[Table 4.5.1]{GLS3}.

Suppose that $r \ge 5$.  Then $G$ has cyclic Sylow $3$-subgroups and $3$ divides $r^a -\delta $ for some $\delta\in\{\pm 1\}$. We know $G$ contains a dihedral subgroup $D\cong \Dih(2(r^a-\delta))$ and we may suppose that $d \in D$. Since $D$ is generated by involutions and $D \not \le K$, there is a conjugate of $z$ in  $D$. As this is impossible in $G$, we deduce that $r=3$. Since $G \cong \PGL_2(3^a)$, all the $3$-elements of $G$ are conjugate. Assume that $a=2^bm$ where $m>1$ is odd.  From the list of maximal subgroups in $K$, we see that $K$ does not contain   subgroups isomorphic to $\PGL_2(3^{2^b})$ and so, as this   group is a subgroup of $G$, and contains involutions not in $K$, this   cannot happen. Hence $a= 2^b$. That is $G \cong \PGL_2(3^{2^b})$.

We now show that $G\cong \PGL_2(3^{2^b})$ has an invariable $(2,3)$-sequence $(z,d)$.  Let $z \in G\setminus K$ and $d \in K$ be elements of order $2$ and $3$ respectively.  We consider $(z,d)$ and note that it suffices to prove that $G=\gen{z,d}$. Pick $U \in \Syl_3(K)$ with $d \in U$, and define $N= N_G(U)$. Then $N$ acts transitively on the non-trivial elements of $U$, $|N/U|=3^{2^b}-1$ and so $G$ has one conjugacy class of elements of order $3$ and $K$ has two.
As $|N\cap K| $ is even, $d$ and $d^{-1}$ are conjugate in $K$ and so
 $\gen{d}$ and $\gen{d^z}$ are not conjugate in $K$.  Set $X= \gen{d, d^z}$. Then $X$ has Sylow $3$-subgroups which are elementary abelian of order at least $9$. Using Dickson's theorem \cite[Theorem 6.5.1]{GLS3}, we have that $X=\gen{d,d^z}$ is either a subfield subgroup or a subgroup of a Borel subgroup   of $K$. In the latter case, as the intersection of distinct Borel subgroups has order coprime to $3$, $\gen{d,d^z} \le U$. But then $\gen{d,d^z} \le U \cap U^z$ implies $U=U^z$ and  we now have   $z \in N$. However, $N/U$ is cyclic of order divisible  by $4$ and so the unique involution of $N/U$ is in $(N\cap K)/U$, a contradiction as $z \not \in K$.
We conclude that  $X$ is a subfield subgroup realized over a field of order $3^{2^c}$. If $c< b$, then $N_K(X) \cong \PGL_2(3^{2^c})$ and so $d$ and $d^z$ are conjugate in $K$, a contradiction. Hence $b=c$ and $G=\gen{z,d}$. This shows that $G$ is invariably $(2,3)$-generated.
\end{proof}

 \section{ Groups of Lie type}\label{sec:7}

 In this section, we carry on  the proof of \cref{thm:Lietype23as} until its conclusion. Suppose that $G$ is a counterexample to  \cref{thm:Lietype23as}. Because of \cref{lem:alt gone,lem:sporadic gone,lem:psl2} we may assume that $K$ is a Lie type group defined in characteristic $r$ (including ${}^2\F_4(2)'$) and that $K \not \cong \PSL_2(r^a)$.

We use the descriptions of automorphisms of $K$ as given  in \cite[Section 2.5]{GLS3}. In particular, $\Phi_K$ is the subgroup of standard field automorphisms and $\Gamma_K$ is the subgroup of standard graph automorphisms. Note that typically for $\gamma \in \Gamma_K$ the coset $\gamma K$ contains more than one $\Aut(G)$-conjugacy class of elements of order the same as $\gamma$.

 \begin{lemma}\label{not both field or graph field} If $z,d \in G\setminus \Inndiag (K)$. Then at least one of $z$ or $d$ is a graph  automorphism.
 \end{lemma}
 \begin{proof} If neither $z$ nor $d$ is a graph automorphism, then both $z$ and $d$ must be $\Aut(K)$-conjugate to field or graph-field automorphisms by \cite[Propositions 4.9.1]{GLS3}.
 Since field and graph-field automorphisms normalize Borel subgroups of $K$ by \cite[Theorem 2.5.1 c]{GLS3}, and all Borel subgroups of $K$ are $K$-conjugate, we may assume that  $G=\gen{z,d}\le N_G(B)<G$, a contradiction.
 \end{proof}
\subsection{Small rank groups}

\begin{lemma}\label{lem:notpsl3} We have $K\not \cong \PSL_3^\epsilon(r^a)$ where $\epsilon =\pm 1$.
\end{lemma}

\begin{proof} Set $q=r^a$.
 Then $K$ has one conjugacy class of involutions by \cite[Table 4.5.1]{GLS3} and \cite[(4.1) and (6.1)]{AschbacherSeitz} and if $z$ is an outer automorphism,  \cite[Table 4.5.1 and Proposition 4.9.2]{GLS3} imply that $z$ is $\Aut(K)$-conjugate to a   field, graph-field or graph automorphism.

 We use the subgroup structure of $K$ as described in \cite[Theorem 6.5.3]{GLS3}.

Since $K$ has one conjugacy class of involutions, $z $ centralizes an involution in $K$ and so $|C_K(d)|$ is odd.  In particular, $d$ is not a field automorphism and so $d \in \Inndiag(K)$.

Since every involution in $G$ normalizes a Borel subgroup of $K$, $d$ does not normalize a Borel subgroup of $K$. In particular, $q\ne 3^a$.

If $3$ divides $q-\epsilon$, then the image $M$ in $\PGL_3^\eps(q)$ of the subgroup  $(q-\eps)^3:\Sym(3)$ of monomial matrices contains a Sylow $3$-subgroup of $\PGL_3^\eps(q)$. This subgroup is normalized by the inverse transpose and field automorphisms. Therefore, as $M \cap K$ has even order, we have $z^G\cap N_G(M\cap K)\not= \emptyset \not= d^G\cap N_G(M\cap K)$, a contradiction.

Therefore $3$ divides $q+\eps$. In this case the Sylow $3$-subgroups of $G$ are cyclic and are conjugate into a subgroup isomorphic to $\SL_2(q)$. As $|C_K(d)|$ is odd, $q$ is a power of $2$ and so $z$ is an outer automorphism.  Using \cref{lem:cop3} we have $|C_K(z)|$ is divisible by $3$, a contradiction as $K$ has cyclic Sylow $3$-subgroups. We conclude that  $G$ is not invariably $(2,3)$-generated.
\end{proof}

\begin{lemma}\label{lem:smallRee}
We have that $K \not \cong {}^2\G_2(3^a)$, $a \ge 3$.
\end{lemma}

\begin{proof}   In this case $a$ is odd and $|G/K|$ divides $a$. Hence $G$ has one conjugacy class of involutions and so $z \in K$.  Now the Borel subgroup  $B$ of $G$ has even order. In particular, we may assume that  $z\in B$. As $d$ has order $3$, whether $d$ is in $K$ or $d$ is a field automorphism, we may conjugate $d$ to normalize $B$, a contradiction.
\end{proof}

\subsection{General lemmas}\label{subsec:gen lemmas}

\begin{lemma}\label{lem:Ckd even} We have $|C_K(d)|$ is even.
\end{lemma}

\begin{proof} Assume the contrary. Then $|C_K(d)|$ is odd. Now  \cite[Theorem 3]{Gerhardt} provides us with the list of possible pairs $(K,d)$. Using \cref{lem:psl2,lem:notpsl3,lem:smallRee,lem:alt gone,lem:sporadic gone}, we are left with
$K$ being one of the following groups: $\PSL_4(3^a)$, $a$ odd, $\PSU_4(3^a)$ and  $\G_2(3^a)$ with $a \ge 1$ and in each case $d \in K$. In particular, $d$ is contained in a conjugate of every parabolic subgroup of $K$. By \cref{lem:cop3},   $z$ centralizes an element of order $3$ and thus normalizes a parabolic of $K$ by \cref{lem:normalise a parabolic},  a contradiction.
\end{proof}

\begin{lemma}\label{lem:z not in K} We have that  $z\not \in K$.
\end{lemma}

\begin{proof} Suppose that $z \in K$. Then $G= K\gen{d}$. By \cref{lem:Ckd even}, $|C_K(d)|$ is even. Let $u \in C_K(d)$ be an involution. Then by \cref{GR}, there exists a broad subgroup $A$ with $u \in A \cap C_K(d)$. As $A$ is broad, we may replace  $z$  by a conjugate so that  $z \in A$. But then $\gen{z,d} \le C_G(u)<G$, a contradiction.
\end{proof}

\begin{lemma}\label{lem no 3-broad}
If $K$ has a $3$-broad subgroup and $K \not \cong \Sp_4(2^a)$, then $d\not \in K$.
\end{lemma}

\begin{proof} For a contradiction suppose that $d \in K$. We have that $z \not \in K$ by \cref{lem:z not in K}. Using \cref{lem:cop3} we have $C_K(z)$ has order divisible by $3$. Let $E$ be a $3$-broad subgroup containing $d$. Then there exists $e \in C_K(z)$ of order $3$ and $g\in G$  such that $e^g \in E$. But then $G=\gen{z^g,d} \le C_G(e^g)$, a contradiction. Hence $d \not \in K$.
\end{proof}

\begin{lemma}\label{lem:coprime r}
We have  $d$ is not a field automorphism and, if   $r=3$, $d\not \in K$.
\end{lemma}

\begin{proof} Suppose that contrary. Then
$d$ normalizes a representative of each class of parabolic subgroups. By \cite[Theorem 4.5.1 and Proposition 4.9.1]{GLS3}, $r$ divides  $|C_K(z)|$. Now \cref{lem:normalise a parabolic} shows that $z$ normalizes a parabolic subgroup of $G$, and this is impossible.
\end{proof}

\begin{proposition}\label{prop:not char 3}
 We have $r \ne 3$.
\end{proposition}

\begin{proof} Assume that $r=3$.   Then \cref{lem:coprime r} implies that $K \cong \mathrm P\Omega_8^+(3^a)$ and $d$ is  a graph or graph-field automorphism. Now \cref{lem:NGD struct,lem:lastlem} provide  a contradiction.
\end{proof}

\subsection{Classical groups in characteristic $2$}

In this subsection, we complete the investigation of classical groups defined in characteristic $2$.

We begin by eliminating a potentially troublesome case.

\begin{lemma}\label{lem:not PSp4}  We have $K \not \cong \PSp_4(2^a)$ with $a \ge 2$.
\end{lemma}

\begin{proof}   We know that $z \not \in K$ by \cref{lem:z not in K}. Hence $z$ is a field or a graph-field automorphism as described in \cite[Theorem 2.5.1]{GLS3}. Since $K=\Inndiag(K)$, \cref{lem:coprime r} gives $d \in K$. Fix $H \le K$, with $H\cong  \Sp_4(2)$ such that $H$ is normalized by $z$. Let $W$ be a Sylow $3$-subgroup of $H$, then $W= \Omega_1(S)$ for some Sylow some  Sylow $3$-subgroup of $K$. Hence $d$ is conjugate to an element of $H$ and we have a contradiction.
\end{proof}

The next lemma goes beyond the remit of this section.

\begin{lemma}\label{lem:lem char 2}
Assume that $r=2$ and that $K \not \cong {}^2\F_4(2)'$.
 Then  $K\cong \PSL_n(2^a)$, $\POmega_{2n}^+(2^a)$, $\F_4(2^a)$ or $\E_6(2^a)$ and $z$ is a graph or  a graph-field automorphism.
\end{lemma}

\begin{proof}   By  \cref{lem:notpsl3,lem:not PSp4}, $K\not\cong  \PSL_3^\epsilon(2^a)$, or $\PSp_4(2^a)$ (or $\G_2(2)'$ or $\PSp_4(2)'$).

 By \cite[Theorem 1]{Gerhardt}, $d$ normalizes a non-trivial $2$-subgroup $R$ of $K\gen{d}$.
  Therefore  \cref{lem:normalise a parabolic} implies that  $d$ normalizes a parabolic subgroup $L$ of $G$.  If $z \in \Inndiag(K)\Phi_K$,  then $K\gen{z}=N_{K\gen{z}}(L)K$ by \cref {lem:factors}.
  Therefore $N_{K\gen{z}}(L)$ contains a Sylow $2$-subgroup of $K\gen{z}$ and so contains a conjugate of $z$.
  But $N_G(L)$ also contains $d$, this contradicts $(z,d)$ being $(2,3)$-sequence for $G$.
  Thus $z \not \in \Inndiag(K)\Phi_K$. This proves the claim.
\end{proof}

\begin{lemma}\label{lem:not psln}
We have $K \not \cong \PSL_n(2^a)$ with $n \ge 4$.
\end{lemma}

\begin{proof} By \cref{lem:lem char 2}, $z$ is a graph or  graph-field automorphism and \cref{lem:coprime r} implies $d\in \Inndiag(K)$.

If $3$ does not divide  $(n,2^a-1)$, then $d \in K$ and $K$ has a $3$-broad subgroup by \cref{lem:3good} which is impossible by \cref{lem no 3-broad}.

Hence $3$  divides $(n,2^a-1)$. If the preimage $\wh d$  of $d$ in $\GL_n(2^a)$ has order $3$, then $\wh d$ diagonalizes and so $d$ is contained in a Borel subgroup of $G$. But then $d$, and $z$ normalize a Borel  subgroups of $G$, a contradiction.  Hence $\wh d$ has order at least $9$ and $\wh {d}^3 \in Z(\GL_n(2^a))$.  It follows that $\wh d$ is conjugate to an element of  $L$ where $L$ is the intersection of $\GL_3(2^a) \times \dots\times \GL_3(2^a)$ with $ \SL_n(2^a)$. Now, as $n \ge 4$, the image of  $L$ in $\PGL_n(2^a)$, is contained in a Levi complement of a proper parabolic subgroup $P$  which is invariant under the standard graph automorphism
$\gamma$ and field automorphism $\phi$. Hence $P\gen{\phi,\gamma}$ contains a Sylow $2$-subgroup of $\Aut(K)$ and hence contains a conjugate of $z$.
  This shows that $G$ is not invariably $(2,3)$-generated.
\end{proof}

\begin{lemma}\label{lem:not Omega char 2}
We have $K \not \cong \mathrm P\Omega_{2n}^+(2^a)$ with $n \ge 4$.
\end{lemma}

\begin{proof} We know that $z$ is a graph or graph-field automorphism  by \cref{lem:lem char 2}  and, by \cref{lem:coprime r}, $d$ is not a field automorphism. Notice that $z$ normalizes a Sylow $2$-subgroup of $K$ and so $z$ normalizes a Borel subgroup $B$ of $K$. By \cref{lem:3good} and \cref{lem no 3-broad}, we know $d \not \in K$. As $\mathrm{Outdiag}(K)=1$,  $K \cong \mathrm{P}\Omega_8^+(2^a)$. Let $\gamma$ be the standard triality automorphism with respect to $B$. Then $N_G(B) \ge \gen{z,\gamma}$. If $d$ is a graph-field  automorphism, then $d$ normalizes some Borel subgroup of $K$ and so we can conjugate $d$ into $N_G(B)$, a contradiction. Thus we may assume that $d$ is in the coset  $\gamma K$. In particular, $\gamma\in G$. Since $\Out(K) \cong \Phi_K\times \Gamma_K$, $z$ is a graph or a graph-field automorphism and $G/K$ is invariably $(2,3)$-generated, we have $$\Sym(3)\cong G/K = N_G(B)K/K\cong N_G(B)/B$$ independently of whether $z$ acts as a graph or a graph-field automorphism.

If $d$ is conjugate to $\gamma$  then we can conjugate $d$ into $N_G(B)$, a contradiction.

Therefore,   $d$ is $K$-conjugate to $\gamma_2$
 as in \cite[Table 4.7.3A]{GLS3}. We see that $d$ centralizes a non-trivial $2$-subgroup of $K$ and therefore $d \in N_G(P)$ for some parabolic subgroup $P$ of $K$ by \cref{lem:normalise a parabolic}.  We  may conjugate $d$ and $P$ so that $P \ge  B$. So now $d\in N_G(P)$ and $P\ge B$.  Then $N_G(P)=N_{N_G(P)}(B)P$ as $P=N_K(P)$. As $d \in N_G(P)$, $3$ divides $|N_{N_G(P)}(B):N_K(B)|=|N_{N_G(P)}(B):B|$ by the Frattini argument.
 As $N_G(B)=B\langle \gamma,z\rangle$, and $N_G(B)/B \cong \Sym(3)$, we have $\gamma \in N_{N_G(P)}(B)$. Therefore $P$ is a $\gamma$-invariant parabolic subgroup of $K$. The $\gamma$-invariant parabolic subgroups containing $B$ are $B$, the minimal parabolic subgroup corresponding to the middle node of the Dynkin diagram, and the maximal parabolic subgroup corresponding to the end nodes of the $\mathrm D_4$-Dynkin diagram. These parabolic subgroups are all normalized by $N_G(B)$. Thus $z\in N_G(B) \le N_G(P)$ and so $G=\gen{z,d} \le N_G(P)<G$, a contradiction.
\end{proof}

We summarise the conclusions of this subsection as follows.

\begin{proposition}\label{prop:not classical 2} Suppose that $G$ satisfies the hypothesis of \cref{thm:Lietype23as}. If $K$ is a simple classical group defined in characteristic $2$, then $K \cong \PSp_4(2)'\cong \PSL_2(9)\cong \Alt(6)$.
\end{proposition}

\begin{proof} This follows by combining  \cref{lem:lem char 2,lem:psl2,lem:notpsl3,lem:not PSp4,lem:not psln,lem:not Omega char 2}.
\end{proof}

\subsection{Exceptional groups}

In this subsection we prove the following proposition.

\begin{proposition}\label{prop:not excep}  Suppose that $G$ satisfies the hypothesis of \cref{thm:Lietype23as}. Then $K$ is not an exceptional simple group of Lie type.
\end{proposition}

By \cref{prop:not char 3},   $K$ is not an exceptional group characteristic $3$.

\begin{lemma}\label{lem:suzuki or large ree}
We have $K \not\cong {}^2\B_2(2^a)$, $a \ge 3$ odd or ${}^2\F_4(2^a)'$ with $a\ge 1$ odd.
\end{lemma}

\begin{proof}  If $a \ge 3$, this is included in \cref{lem:lem char 2}. If $K \cong {}^2\F_4(2)'$,   then $z \in G\setminus K$ by \cref{lem:z not in K}. But $G$ has no such elements of order $2$ \cite[Theorem 3.3.2 (d)]{GLS3}.
\end{proof}

\begin{lemma}\label{lem:not G2 or 3D4}  We have  $K \not \cong {}^3\D_4(r^a)$ or $\G_2(r^a)'$.
\end{lemma}

\begin{proof}  If $K \cong \G_2(2)'\cong \PSU_3(3)$,  \cref{lem:notpsl3} gives a contradiction.

Otherwise, by \cref{lem:lem char 2},  $r$ is odd. Thus $K$ has a unique conjugacy class of involutions by \cite[Table 4.5.1]{GLS3}. Then  $z$ commutes with an involution, $u\in K$  and, by \cref{lem:Ckd even}, $|C_K(d)|$ is even. Hence $d$ is conjugate to an element of $C_G(u)$, a contradiction.
\end{proof}

\begin{lemma} \label{lem:not E7 E8 or F4} We have $K \not \cong \F_4(r^a)$, $\E_7(r^a)$ or $\E_8(r^a)$.
\end{lemma}

\begin{proof}   We have $z \not \in K$ by \cref{lem:z not in K}, and \cref{3-broad,lem no 3-broad} give $d \not \in K$. It follows that $d$ is a field automorphism, contradicting \cref{lem:coprime r}.
\end{proof}

  \begin{lemma}\label{lem:not E6}  We have  $K \not \cong \E_6(r^a)$ or ${}^2\E_6(r^a)$.
  \end{lemma}

\begin{proof} We write $K=\E_6^\eps(r^a)$ where $\eps =\pm 1 $ and handle both cases simultaneously. We have $z \not \in K$ by \cref{lem:z not in K} and $d \in \Inndiag(K)$ by \cref{lem:coprime r}.

Suppose that $3$ does not divide $r^a-\eps$.
Then $K$ has a $3$-broad subgroup  by \cref{lem:E6 good cases}.   Hence \cref{lem no 3-broad} says that $d \not \in K=\Inndiag(K)$, a contradiction. Therefore $3$  divides $r^a-\eps$.

Assume that $r$ is odd. Then $r \ge 5$ and we apply \cref{lem:E6 big 2-part} to see that $|K:C_K(z)|_2= (q-\eps)^2_2$  and $|C_K(d)|_2> (q-\eps)^2_2$.  Let $S \in \Syl_2(K)$ and  replace  $z$ and $d$ by $K$-conjugates if necessary  so that $S_d=S \cap C_K(d) \in\Syl_2(C_K(d))$ and $S_z=S \cap C_K(z) \in \syl_2(C_K(z))$. Then  \cref{lem:E6 big 2-part} implies $|S:S_z|=  (q-\eps)^2_2$  and $|S_d|> (q-\eps)^2_2$. Hence $S_z\cap S_d \ne 1$ and $\gen{z,d}\le C_G(S_z\cap S_d)$, a contradiction. Thus $r=2$.

 By \cref{lem:lem char 2},
 $K\cong \E_6(2^a)$  and $z$ induces a graph or a graph-field automorphism on $K$. (So as $3$ divides $2^a-1$, $\epsilon =1$ coincides with the notation in \cite[Table 4.7.3A]{GLS3}.)
   We know $3$ divides $|C_K(z)|$ from \cref{lem:cop3}.  Take $t\in K$ to be a $3$-element of type $t_3$ as in \cite[Table 4.7.3A]{GLS3}. Then $t$ is a $3$-central element of $K$ as the centre of a Sylow $3$-subgroup of its centraliser has order $3$ (which is also the centre of  $C_K(t)$).
Now the  standard graph automorphism $\gamma$ of $K$ normalizes $L=E(C_K(t))=K_1\circ K_2\circ K_3$ with each $K_i\cong \SL_3(2^a)$ and $K_1=\langle  X_{\pm\alpha}\mid  \alpha\in\{\alpha_4,-\alpha_*\}\rangle$, $K_2=\langle X_{\pm\alpha_i}\mid i=1,2\rangle$ and $K_3=\langle X_{\pm\alpha_i}\mid i=5,6\rangle$. Then $L$ is normalized by $\gamma$, by $\Phi_K$ and by any element of $\gamma X_{-\alpha_*}$. In particular,   a conjugate of $z$ is in $N_G(L)$. Since $N_G(L)$ contains a Sylow $3$-subgroup of $G$, we conclude that $G$ is not invariably $(2,3)$-generated.
\end{proof}

Finally \cref{lem:smallRee,lem:suzuki or large ree,lem:not G2 or 3D4,lem:not E7 E8 or F4,lem:not E6} show that there are no   counter examples  to  \cref{prop:not excep}.

\subsection{The classical groups   for $r\ge 5$}

In this final subsection, we consider the possibility that $K$ is a classical group defined in characteristic $r\ge 5$ and then complete the proof of \cref{thm:Lietype23as}.

\begin{lemma}\label{GLGone} We have  $K \not \cong \PSL_n^\eps(r^a)$, with $\eps \in \{\pm 1\}$, $n \ge 4$ and $r\ge 5$.
\end{lemma}

\begin{proof} Set $q=r^a$.  \cref{lem:coprime r} implies  $d\in \Inndiag(K)$.

If $(q-\epsilon,3)=1$, then by \cref{lem:3good}, $G$ has a $3$-broad subgroup, and so by \cref{lem no 3-broad}, $d\not\in K=\Inndiag(K)$, a contradiction.
Hence, $q\equiv \eps\pmod 3$.

Define $\gamma$ to be the graph automorphism given by inverse transpose. Let   $H$  be the image of $\GL^{\epsilon}_1(q) \wr \Sym(n)$ in $\Aut(K)$ as in \cref{GLfact}.  Then $H$ is normalized by $\Phi_K$ and $\gamma$. Let $H^*= N_{\Aut(K)}(H)$. Then $\Aut(K)=H^*K$ and $H^*$ contains a Sylow $3$-subgroup of $\Aut(K)$.

If  $q\equiv \epsilon \pmod 4$, then $H$ contains a Sylow $2$-subgroup of $\Inndiag(K)$.
Thus \cref{lem:contain r and s subgroups}  with $ H^*$ in place of $H$ in that lemma, shows that $G$ is not invariably $(2,3)$-generated. So suppose that  $q\equiv -\epsilon \pmod 4$. Then $H^*$ contains a Sylow $3$-subgroup and a representative of every involution class in $\PGL_n^\eps(q)$ by \cref{GLfact} (iii).
Thus if $\epsilon=1$,  $z$  is either a graph   a graph-field  or a field automorphism, and if $\epsilon=-1$, then following \cite{GLS3}, $z$ is a graph automorphism. By   \cite[Table 4.5.1]{GLS3} we may conclude that one of the following holds: $\epsilon=1$ and $z$ is contained in $\gen{\gamma,\Phi_K}$, $\epsilon=-1$ and $z =\gamma$, or that $\epsilon\in\{\pm1\}$, $n$ is even and   $z$ is one of  $\gamma_1$, $\gamma_2$ or $\gamma_2'$. Explicit descriptions of $\gamma_1$, $\gamma_2$ and $\gamma_2'$ are given in \cite[Section 3.2.5 and Proposition 3.2.11]{BurnessGiudici}. From these descriptions we see that we may replace $z$ with a $\Aut(K)$-conjugate in $H^*$.
Now $\Aut(K)= H^*K$ and so applying \cref{lem:lastlem} with $D= H^*\cap K$, shows that $G$ is not invariably $(2,3)$-generated.
 \end{proof}

\begin{lemma}\label{lem:orth gone} Suppose that $r \ge 5$. Then   $K\not \cong \PSp_{2n}(r^a)$ with $n \ge 2$ and
 $K \not \cong \mathrm P \Omega_{2n}^\eps(r^a)$ or $\mathrm P\Omega_{2n-1}(r^a)$ with $n \ge 4$, $\eps\in \{\pm 1\}$.
\end{lemma}

\begin{proof} We know from \cref{GR,lem:3good} that $K$ has a broad subgroup and a $3$-broad subgroup. \cref{lem:z not in K,lem no 3-broad} show that $\{z,d\}\subset G\setminus K$. By \cref{lem:coprime r}, $d$ is not a field automorphism. As $\mathrm{Outdiag}(K)$ is a $2$-group, $d$ must be a graph or a graph field automorphism and $K=\mathrm P\Omega_8^+(r^a)$. Now we apply \cref{lem:NGD struct,lem:lastlem} to obtain a contradiction.
\end{proof}

We can now prove  \cref{thm:Lietype23as}.
\begin{proof}[The proof of \cref{thm:Lietype23as}] Let $G$ and $K$ be as in the statement of the theorem.
 If $K \cong \PSL_2(r^a)$, then \cref{lem:psl2} yields that $G\cong \PGL_2(3^{2^b})$ for some $b\ge 1$. Avoiding $\Alt(6) \cong \PGL_2(9)$, \cref{lem:alt gone,lem:sporadic gone} show that $K$
  is not a sporadic simple group or an alternating group of degree at least $7$ or $5$.

  \cref{prop:not char 3} shows that $K$ is not a Lie type group defined in characteristic $3$ other than the cases which arise from $\PSL_2(3^{2^b})$ with $b \ge 1$.  Similarly, \cref{prop:not classical 2} shows that among the classical groups defined in characteristic $2$ only  $\PSp_4(2)' \cong \PSL_2(9)\cong \Alt(6)$ is an exception.

\cref{prop:not excep} shows that $K$ is not an exceptional group. In particular, we may assume that   $K$ is not a Lie type group defined in characteristic $2$.
Therefore $K$ is a classical group defined in characteristic $r \ge 5$. \cref{GLGone,lem:orth gone}  eliminate these final cases. Thus if $G$ is invariably $(2,3)$-generated then $G \cong \PGL_2(3^{2^b})$ for some $b \ge 1$.  \cref{lem:psl2} also shows that these groups are invariably $(2,3)$-generated. This completes the proof.
\end{proof}

\section{The proof of Theorem C}\label{sec:8}

This final short section provides our proof of Theorem C.
We start with the following lemma.

\begin{lemma}\label{lem:no repeated factors}
Suppose that $H \cong \PGL_2(3^{2^a})$ and   $X= H\times H$. Then no subgroup of $X$ containing $F^*(X)$ is invariably $(2,3)$-generated.
\end{lemma}

\begin{proof} Set $q=3^{2^a}$,   $H_1=\{(h,1)\mid h \in H\}$, $H_2=  \{(1,h)\mid h \in H\}$, $K_i = H_i \cap F^*(X)$ and define $K=K_1K_2=F^*(X)$.
Suppose that $G$ is invariably $(2,3)$-generated with $K \le G \le X$.

 For $1\ne N<K$, a  normal subgroup of $G$, we have $N= K_1$, or $N= K_2$.  Hence $G/N$ is isomorphic to a subgroup of $2 \times \PGL_2(q)$.  By \cref{lem:quotients}, $G/N$ is invariably $(2,3)$-generated as are all its quotients. Since the elementary abelian subgroup of order $2^2$ is not invariably $(2,3)$-generated,  $G \ne X$. By  \cref{lem:psl2}, $\PSL_2(q)$ is not invariably $(2,3)$-generated and so $G/N \not \cong 2 \times \PSL_2(q)$. It follows that $G/N \cong \PGL_2(q)$ and that $|X:G|=2$.

Assume that $(\zeta,\delta)$ is an invariable $(2,3)$-sequence for $G$. Then $\zeta \not \in K$.

   We know $G = K\gen{(z,z)}$ with $z\in H\setminus F^*(H)$ of order $2$. Then $\zeta=(z_1,z_2)$ with $z_1$ and $z_2$ of order $2$ and conjugating by elements of $K_1$ and $K_2$ we see that $\zeta$ is $G$-conjugate to $(z,z)$. Hence we may assume $\zeta=(z,z)$.

      Let $D \in \Syl_3(G)$ contain $\delta$, $D_1= D\cap H_1$ and $D_2=D\cap H_2$. Then $D=D_1D_2$ and $N_G(D)$ has order $q^2(q-1)^2/2$ with $N_G(D)/D$ isomorphic to the direct product of  cyclic groups of orders $q-1$ and $(q-1)/2$. Thus $N_G(D  )$ has four conjugacy classes of elements of order $3$: one class in $D_1$, one class in $D_2$ and two classes on the diagonal of $D_1D_2$ each of length $(q-1)^2/2$.

      If $\delta \in D_i$ with $i=1$ or $2$,  then $\langle D_i,\zeta \rangle \le H_i\gen{\zeta} < G$, a contradiction as $\gen{\zeta,\delta}= G$.   Hence $\delta$ is not in $D_1$ or $D_2$.

      Since $N_G(D)$ controls $3$-fusion in $G$ by Burnside's fusion theorem, $G$ has two conjugacy classes of elements of order $3$  diagonal between $H_1$ and $H_2$. The representatives are $(d,d)$ and $(d,d^z)$ where $d$ in $H$ has order $3$. Indeed, suppose that $g=(az^\ell,bz^\ell) \in G$ with $\ell \in \{0,1\}$ and $a,b \in H'$ is such that $(d,d^z)=(d,d)^g$.
      Then $(d,d^z) = (d^{az^\ell},d^{bz^\ell})$ implies that $d= d^{az^\ell}$ from which we conclude that $z^{\ell}=1$. But then $d^z=d^b$ with $b \in H'$, a contradiction as $d$ and $d^z$ are not conjugate in $H'$.

   Set $L= \{(h,h) \mid h \in H\}< G$ and $M= \{(h,h^z) \mid h \in H\}<G$. Then $L$ and $M$ both contain $\zeta $. Since $(d,d) \in L$ and $(d,d^z) \in M$, we have no candidates for $\delta$. This proves the claim.
\end{proof}

The following example, computed using {\sc magma}, shows that direct products of isomorphic simple groups can in some cases be invariably $(s,t)$-generated for $s$ and $t$ primes. To see it theoretically, just note that we can select an element of order $31$ which is not in any diagonal subgroup isomorphic to $\PSL_5(2)$.

\begin{example} $\PSL_5(2)\times \PSL_5(2)$ is invariably $(2,31)$-generated.
\end{example}

We at last prove Theorem C.

\begin{proof}[Proof of Theorem C] Suppose that $G$ is invariably $(2,3)$-generated. If  $G$ is  soluble, then  $G$ is a $\{2,3\}$-group by \cref{lem:Sylow/Hall}. So assume that $G$ is not soluble.

We first study the almost semisimple quotient $\overline G=G/A(G)$. By Theorem A, $\ov G$ is non-trivial and $\ov G$ is determined by its almost simple quotients by \cref{prop: G/A(G)}. Assume that $J,K \in \mathcal A^*(G)$ with $J\ne K$. We have $G/J\cong \PGL_2(3^{2^b})$ by \cref{thm:Lietype23as}.
If  $\Soc(G/J)\cong \Soc(G/K)$, then, setting $H=G/J$, $G/(J\cap K)$ is isomorphic to a subgroup of $X=H \times H$ containing $F^*(X)$. But then \cref{lem:no repeated factors} shows that $G/(J\cap K)$ is not invariably $(2,3)$-generated and this contradicts \cref{lem:quotients}. Thus all the almost simple quotients of $G$ have pairwise non-isomorphic socles. Let $M$ be the full preimage of  $\Soc(\ov G)$. Then $G/M$ is an elementary abelian $2$-group which is invariably generated by an involution. Hence $|G/M|=2$ and this gives the structure of $\ov G$ described in Theorem C (i) and (ii).

 For the structure of $A(G)$, we apply Theorem A just noting that every divisor of $|\Soc(\ov G)|$ which is coprime to $3$ is   in   $\pi(|\PSL_2(3^{2^{a_n}})|_{3'})=\pi(3^{2^{a_n}}-1)\cup \pi(3^{2^{a_n}}+1)$. Finally, \cref{thm:Lietype23as}   shows that the almost simple groups $\PGL_2(3^{2^b})$ are indeed invariably $(2,3)$-generated.
 Finally,  \cref{lem:dif socles} shows that for $1  <a_1 <\dots< a_n$, there are invariably $(2,3)$-generated groups with socle $\PSL_2(3^{2^{a_1}}) \times \dots\times \PSL_2(3^{2^{a_n}})$.
\end{proof}

\printbibliography

\end{document}